\documentclass[11pt,twoside]{article}
\usepackage{amssymb}
\usepackage{amssymb,amsmath,amsthm,amsfonts,mathrsfs,hyperref}
\usepackage{times}
\usepackage{enumerate}
\usepackage{cite,titletoc}

\allowdisplaybreaks

\def\N{\mathop{\mathbb N\kern 0pt}\nolimits}
\def\Q{\mathop{\mathbb Q\kern 0pt}\nolimits}
\def\R{\mathop{\mathbb R\kern 0pt}\nolimits}
\def\SS{\mathop{\mathbb S\kern 0pt}\nolimits}

\hypersetup{colorlinks=true,linkcolor=blue,citecolor=red,urlcolor=cyan}

\theoremstyle{plain}
\newtheorem{theorem}{Theorem}[section]
\newtheorem{proposition}[theorem]{Proposition}
\newtheorem{lemma}[theorem]{Lemma}

\theoremstyle{definition}

\numberwithin{equation}{section}

\title{The necessity theory for commutators of multilinear singular integral operators: the weighted case}

\author{Dinghuai Wang \footnote{Dinghuai Wang(\texttt{Wangdh1990$@$126.com}) are supported by
     National Natural Science Foundation of China(No.11971237,12071223), the Natural Science Foundation of the Jiangsu Higher Education Institutions of China(No.19KJA320001) and Doctoral Scientific Research Foundation.}\\
    [12pt] {\small School of Mathematics and Statistics, Anhui Normal University, Wuhu, 241002, China}\\}

\begin{document}

\date{}
\maketitle
\thispagestyle{empty}

\begin{abstract}
In this paper, the necessity theory for commutators of multilinear singular integral operators on weighted Lebesgue spaces is investigated. The results relax the restriction of the weights class to the general multiple weights, which can be regarded as an essential improvement of \cite{ChafCruz2018,GLW2020}. Our approach elaborates on a commonly expanding the kernel locally by Fourier series, recovering many known results but yielding also numerous new ones. In particular, we answer the question about the necessity theory of the iterated commutators of the multilinear singular integral operators.

\vskip 0.2 true cm


\vskip 0.2 true cm
\noindent
\end{abstract}

\vskip 0.6 true cm
\tableofcontents

\section{Introduction}

The foundational paper of Coifman-Rochberg-Weiss \cite{CRW1976} provided a constructive proof of the weak factorizations of the classical Hardy space $H^1$ in terms of Riesz transforms. The result depends upon the duality between $H^{1}$ and $BMO$ and upon a new result linking $BMO$ and the $L^{p}$ boundedness of certain commutator operators. Let $T$ be one of the Riesz transforms.  If $b$ is in $BMO$, then
the commutator $[b, T](f)=bT(f)-T(bf)$ is bounded on $L^{p}$ for $1<p<\infty$. Conversely, for some $p$ in $(1,\infty)$, the commutator $[b,T]$ is bounded on $L^{p}$, then $b$ is in $BMO$. These estimates have found many important applications in other areas of operator theory and partial differential equations. For example, the investigations of div-curl lemmas \cite{CLMS1993, LPPW2012} and additional interpretations in operator theory
\cite{NPTV2002,N1957} came out of this work.
The theory was then extended and generalized to several directions. For instance, Bloom \cite{Bloo1985} investigated the same result in the weighted setting; Uchiyama extended the boundednss results on the commutator to compactness \cite{Uchi1978} and considered the spaces of homogeneous type \cite{Uchi1981}; Krantz and Li in \cite{KL12001} and \cite{KL22001} have applied commutator theory to give a compactness characterization of Hankel operators on holomorphic Hardy spaces $H^{2}(D)$, where $D$ is a bounded, strictly pseudoconvex domain in $\mathbb{C}^n$. It is perhaps for this important reason that the boundedness of $[b,T]$ attracted one's attention among researchers in harmonic analysis and PDEs.
\medskip

Many authors are interested in multilinear operators (Coifman and Meyer \cite{CM1975, CM1978}, Christ and Journ\'{e} \cite{CJ1987}, Kenig and Stein \cite{KS1999}), it was oriented towards the
study of the Calder\'{o}n commutator. Multilinear Calder\'{o}n-Zygmund operators were introduced by Coifman and Meyer but were not systematically studied for about a quarter century until the appearance of \cite{GT2002}. The boundedness results for commutators with symbols in $BMO$ started to receive attention only a few years ago. P\'{e}rez-Torres \cite{PT2003} first introduced the $i$-th commutator of m-linear Calder\'{o}n-Zygmund operator $T$ and showed that
$[b,T]_i$
is bounded from $L^{p_1}\times\cdots\times L^{p_m}$ to $L^p$ provided that $b\in BMO$, $1<p_1,\cdots,p_m,p<\infty$ with $1/p=1/p_1+\cdots+1/p_m$. Subsequently, in \cite{LOPTT2009} Lerner et al. removed the restriction of that $p>1$ and established the multiple weighted version as well as the weak-type endpoint estimate, and see \cite{AnDuong, ChenWu1} for the non-smooth kernels cases. Iterated commutators of multilinear Calder\'{o}n-Zygmund operators and pointwise multiplication with functions in $BMO$ was studied by \cite{PPTT2014}, which complements and completes the theory developed by Lerner et al. in some sense.
\medskip

For the weighted theory in the multilinear setting, the pioneer work of weighted estimates and commutators in this multilinear setting were studied in \cite{GT22002} and \cite{PT2003}. The initial work of multivariable Rubio de Francia extrapolation theorem was obtained by \cite{CUM} (or \cite{D2011}) for $\vec{\omega}\in A_{\vec{P}}$ with $\omega_{i}\in A_{p_{i}}$. These works treat each variable separately with its own Muckenhoupt class of weights. Thus, it is very interesting to obtain the related result using the multivariable nature of the problem. In the celebrated work \cite{LOPTT2009} Lerner et al. resolved the problems proposed in \cite{GT22002} and \cite{PT2003} and established a theory of weights adapted to the multilinear setting. Recently, the long standing problem of multivariable Rubio de Francia extrapolation theorem for the multilinear Muckenhoupt classes $A_{\vec{P}}$ was showed by Li, Martell and Ombrosi in \cite{LMOarXiv}. In this paper, we will systematically study the necessity theory for multilinear commutators without individual conditions on $\omega_{i}$. The results include the following three aspects:
\medskip

(I) Note that most of the previous proofs for the necessity of bounded commutators is expand the kernel locally by Fourier series. Then, Chaffee and Cruz-Uribe \cite{ChafCruz2018} established the necessity of bounded commutators in a general Banach space structure. The result of Guo, Lian and Wu \cite{GLW2020} relax the restriction of Banach spaces in previous results to quasi-Banach spaces with very weak assumptions on the corresponding kernel. However, in \cite{ChafCruz2018} and \cite{GLW2020}, the weights class are restricted in a narrower class ($\vec{\omega}\in A_{\vec{P}}$ with $\omega_{i}\in A_{p_{i}}$) when they studied the commutators of multilinear singular integral operators on weighted Lebesgue space. Thus, it is a nature problem to obtain the weighted results without the restriction of that $\omega_{i}\in A_{p_{i}}$. After this paper was posted on arXiv, professor Li told me that part of the results had been resolved in \cite{L2020}, then we only consider the characterization of $BMO$ via the weak-$L\log^{+}L$ type weighted boundedness of multilinear commutators.
\medskip

(II) The boundedness result of linear commutators of multilinear Calder\'{o}n-Zygmund operators was shown in \cite{PT2003}. However, The necessity conclusion lasted a long time and the proofs in \cite{Chaf2016} treat each component independently. This might explain why any attempt to obtain the linear commutators has been unsuccessful in the last years: if we only know the boundedness of the linear commutator$[\Sigma \vec{b},T]$, it is not clear that whether each $[b_{i},T]_{i}$ is a bounded operator. In 2018, the linear characterization result was obtained in \cite{WZTadm}. However, the weighted results are not valied by the methods employed in \cite{WZTadm}. Then, some useful method would be quite necessary.
\medskip

(III) It has been an open question whether iterated commutators can be used to characterize $BMO$? The earlier work we find illustrative to present some functions, which shed light on the characterization result is not valid for multilinear maximal operator in \cite{HW2021}. However, some of the techniques do not apply to Calder\'{o}n-Zygmund operators, mainly because they lack positive kernels. Therefore, it needs some tedious calculations in applications.
\medskip

This paper is organized as follows. In order to promote reading, most of the basic definitions are given in section \ref{notation} and the expert reader can easily skip this part. After establishing some basic Lemmas (see Lemmas \ref{BMO-lem1} and \ref{BMO-lem2}) for the characterizations of $BMO$ space, the main results for the general commutators are presented in Section \ref{the general commutators 1} and Sections \ref{the general commutators 2}. Section \ref{the linear commutators} is concern with the linear commutators, in which some technique lemmas for the necessity of bounded linear commutators will be shown. We provide the examples and which establish that the estimates and the general theorems for the iterated commutators in Section \ref{the iterated commutators}. Finally, Appendix \ref{A} contains some further results for the weights function.

\section{Preliminaries}\label{notation}

Let $|E|$ denote the Lebesgue measure of a measurable set $E\subset \mathbb{R}^n$. Throughout this paper, the letter $C$ denotes constants which are independent of main variables and may change from one occurrence to another. $Q(x,r)$ denotes a cube centered at $x$, with side length $r$, sides parallel to the axes.

\subsection{Muckenhoupt weights.}

As we will work in the weighted setting, we need the notion of weighted $L^p$ spaces:
$L^p (\omega)=L^p(\mathbb{R}^n,\omega dx)$ denotes the collection of measurable functions $f$ on $\mathbb{R}^n$ such that
$$
\|f\|_{L^p (\omega)}:= \left( \int_{\mathbb{R}^n} |f(x)|^p \omega(x) \,dx \right)^{1/p} < \infty.
$$
We recall the definition of $A_{p}$ weight introduced by Muckenhoupt in \cite{M1972}, which give the characterization of all weights $\omega(x)$ such
that the Hardy-Littlewood maximal operator
$$
M(f)(x)=\sup_{Q\ni x}\frac{1}{|Q|}\int_{Q}|f(y)|dy
$$
is bounded on $L^{p}(\omega)$. For $1< p<\infty$ and a nonnegative locally integrable function $\omega$ on $\mathbb{R}^n$, $\omega$ is in the
Muckenhoupt $A_{p}$ class if it satisfies the condition
$$[\omega]_{A_{p}}:=\sup_{Q}\bigg(\frac{1}{|Q|}\int_{Q}\omega(x)dx\bigg)\bigg(\frac{1}{|Q|}\int_{Q}\omega(x)^{-\frac{1}{p-1}}dx\bigg)^{p-1}<\infty.$$
And a weight function $\omega$ belongs to the class $A_{1}$ if
$$[\omega]_{A_{1}}:=\frac{1}{|Q|}\int_{Q}\omega(x)dx\Big(\mathop\mathrm{ess~sup}_{x\in Q}\omega(x)^{-1}\Big)<\infty.$$
We write $A_{\infty}=\bigcup_{1\leq p<\infty}A_{p}$. For $\omega\in A_{\infty}$, there exists $0<\epsilon<\infty$ such that for all cubes $Q$ and all measurable subsets $A$ of $Q$, we have
$$\frac{\omega(A)}{\omega(Q)}\leq C\Big(\frac{|A|}{|Q|}\Big)^{\epsilon}.$$

In 2009, Lerner et al.\cite{LOPTT2009} build a theory of weights adapted to the multilinear setting as follows. Let $m\geq 2$, $1\leq p_{1},\cdots,p_{m}<\infty$ with $1/p=1/p_{1}+\cdots+1/p_{m}$, and $\vec{P}=(p_{1},\cdots,p_{m})$. Given $\vec{\omega}=(\omega_{1},\cdots,\omega_{m})$, set
$$\nu_{\vec{\omega}}=\prod_{i=1}^{m}\omega_{i}^{p/p_{i}}.$$
We say that $\vec{\omega}$ satisfies the $A_{\vec{P}}$ condition if
$$\sup_{Q}\Big(\frac{1}{|Q|}\int_{Q}\nu_{\vec{\omega}}(x)dx\Big)^{1/p}\prod_{i=1}^{m}\Big(\frac{1}{|Q|}\int_{Q}\omega_{i}^{1-p'_{i}}(x)dx\Big)^{1/p'_{i}}<\infty.$$
When $p_{i}=1$, $\big(\frac{1}{|Q|}\int_{Q}\omega_{i}^{1-p'_{i}}\big)^{1/p'_{i}}$ is understood as $(\inf_{Q} \omega_{j})^{-1}$.

\subsection{The Campanato spaces}
Let $0<q<\infty$ and $-n/q<\alpha<1$. A locally integrable function $f$ is said to belong to Campanato space $\mathcal{C}_{\alpha,q}$ if there exists a constant
$C > 0$ such that for any cube $Q\subset \mathbb{R}^n$,
$$\frac{1}{|Q|^{\alpha/n}}\bigg(\frac{1}{|Q|}\int_{Q}|f(x)-f_{Q}|^{q}dx\bigg)^{1/q}\leq C,$$
where $f_{Q}=\frac{1}{|Q|}\int_{Q}f(x)dx$ and the minimal constant $C$ is defined by $\|f\|_{\mathcal{C}_{\alpha,q}}$.

Campanato spaces are a useful tool in the regularity theory of PDEs due to their better structures,
which allows us to give an integral characterization of the spaces of H\"{o}lder continuous functions when $0<\alpha<1$. The Lipschitz (H\"{o}lder) and Campanato spaces are related by the following equivalences:
$$\|f\|_{Lip_{\alpha}}:=\sup_{x,h\in \mathbb{R}^n,h\neq 0}\frac{|f(x+h)-f(x)|}{|h|^{\alpha}}\approx \|f\|_{\mathcal{C}_{\alpha,q}},\quad 0<\alpha<1.$$
The equivalence can be found in \cite{DS1984} for $q=1$, \cite{JTW1983} for $1<q<\infty$ and \cite{WZTmn} for $0<q<1$.

Specially, $\mathcal{C}_{0,q}=BMO$, the spaces of bounded mean oscillation. The crucial property of $BMO$ functions is the John-Nirenberg inequality \cite{JN1961},
$$|\{x\in Q: |f(x)-f_{Q}|>\lambda\}|\leq c_{1}|Q|e^{-\frac{c_{2}\lambda}{\|f\|_{BMO}}},$$
where $c_{1}$ and $c_{2}$ depend only on the dimension. A well-known immediate corollary of the John-Nirenberg inequality as follows:
$$\|f\|_{BMO}\approx \sup_{Q}\frac{1}{|Q|}\Big(\int_{Q}|f(x)-f_{Q}|^{p}dx\Big)^{1/p},$$
for all $1<p<\infty$. In fact, the equivalence also holds for $0<p<1$. See, for example, the work of Str\"{o}mberg \cite{Str1979}(or \cite{HT2019} and \cite{WZTams} for the general case).

In order to obtain the characterized results of commutators on weighted Lebesgue spaces, we need establish some characterization of $BMO$ in terms of weights class. Muckenhoupt and Wheeden \cite{MW1974} first made deep connection between Muckenhoupt weights and $BMO$ function. They proved that a function $f$ is in $BMO$ if and only if it is of bounded mean oscillation with respect to $\omega$ for all $\omega\in A_{\infty}$. That is, for any $\omega\in A_{\infty}$,
$$\|f\|_{BMO}\approx\|f\|_{BMO_{\omega}}:=\sup_{Q}\frac{1}{\omega(Q)}\int_{Q}|f(x)-f_{\omega,Q}|dx.$$
where $f_{\omega,Q}=\frac{1}{\omega(Q)}\int_{Q}f(x)\omega(x)dx$. A general and very meaningful result of Hart and Torres \cite{HT2019} showed that for any $\mu,\omega\in A_{\infty}$,
$$\|f\|_{BMO}\approx\|f\|_{BMO^{p}_{\mu,\omega}}:=\sup_{Q}\Big(\frac{1}{\omega(Q)}\int_{Q}|f(x)-f_{\mu,Q}|^{p}dx\Big)^{1/p}, \quad 0<p<\infty.$$

\subsection{Multilinear Caldron-Zygmund operators}
Recall that $m$-Calder\'{o}n-Zygmund operator $T$ is a bounded operator which satisfies
$$\|T(f_{1},\cdots, f_{m})\|_{L^{p}}\leq C\|f_{1}\|_{L^{p_{1}}}\times\cdots\times\|f_{m}\|_{L^{p_{m}}},$$
for some $1<p_{1},\cdots,p_{m}<\infty$ with $1/p=1/p_{1}+\cdots+1/p_{m}$ and the function $K$, defined off the diagonal $y_{0}=y_{1}=\cdots=y_{m}$ in $(\mathbb{R}^{n})^{m+1}$, satisfies the conditions as follow:

(1) The function $K$ satisfies the size condition.
$$|K(y_{0},y_{1},\cdots,y_{m})|
\leq\frac{C}{\big(\sum_{k=1}^{m}|y_{k}-y_{0}|\big)^{mn}};$$

(2) The function $K$ satisfies the regularity condition. For some $\gamma>0$ and all $1\leq i\leq m$,
if $|y_{i}-y'_{i}|\leq \frac{1}{2}\max_{0\leq k\leq m}|y_{0}-y_{k}|$,
$$|K(y_{0},\cdots,y_{i},\cdots,y_{m})-K(y_{0},\cdots,y'_{i},\cdots,y_{m})|
\leq\frac{C|y_{i}-y'_{i}|^{\gamma}}{\big(\sum_{k=1}^{m}|y_{k}-y_{0}|\big)^{mn+\gamma}}.$$
Then we say $K$ is a $m$-linear Calder\'{o}n-Zygmund kernel. If $x\notin \bigcap_{i=1}^{m}{\rm supp} f_{i} $, then
$$T(f_{1},\cdots,f_{m})(x)=\int_{\mathbb{R}^n}\int_{\mathbb{R}^n}K(x,y_{1},\cdots,y_{m})f_{1}(y_{1})\cdots f_{m}(y_{m})dy_{1}\cdots dy_{m}.$$

In this paper, we will consider the kernel $K(y_{0},y_{1},\cdots,y_{m})$ is actually of the form $K(y_{0}-y_{1},\cdots,y_{0}-y_{m})$, the multilinear Riesz transforms are special examples of this form.

\subsection{Sharp maximal operators}
For $\delta>0$, let $M_{\delta}$ be the maximal function
$$M_{\delta}(f)(x)=M(|f|^{\delta})(x)^{1/\delta}=\bigg(\frac{1}{|Q|}\int_{Q}|f(y)|^{\delta}dy\bigg)^{1/\delta}.$$
and $M^{\sharp}$ be the sharp maximal function (see \cite{FSsharp})
$$M^{\sharp}(f)(x)=\sup_{Q\ni x}\inf_{c}\frac{1}{|Q|}\int_{Q}|f(y)-c|dy\approx \sup_{Q\ni x}\frac{1}{|Q|}\int_{Q}|f(y)-f_{Q}|dy.$$

The classical result of Fefferman and Stein \cite{FSsharp} showed that for $0<p,\delta<\infty$ and $\omega\in A_{\infty}$, then there exists $C>0$ (depending on the $A_{\infty}$ constant of $\omega$) such that
$$\int_{\mathbb{R}^n}(M_{\delta}(f)(x))^{p}\omega(x)dx\leq C\int_{\mathbb{R}^n}(M^{\sharp}_{\delta}(f)(x))^{p}\omega(x)dx$$
for all function $f$ for which the left hand side is finite.

\subsection{Commutators}
We recall the notion of the general commutator, the linear commutator and the iterated commutator. The definitions of the general commutator and the linear commutator were given by P\'{e}rez and Torres in \cite{PT2003}, which coincides with the linear commutator $[b,T]$ when $m=1$. They proved that if $1<p,p_{1},\cdots,p_{m}<\infty$ and $1/p=1/p_{1}+\cdots+1/p_{m}$, then
$$b_{1},\cdots,b_{m}\in BMO\Longrightarrow[\Sigma\vec{b},T]: L^{p_{1}}\times \cdots \times L^{p_{m}}\rightarrow L^{p}.$$
Subsequently, in the work \cite{LOPTT2009} Lerner et al. removed the restriction of that $p>1$ and established the multiple weighted version as well as the weak-type
endpoint estimate.

Suppose $T$ is a $m$-linear operator and $\vec{b}=(b_{1},\cdots,b_{m})$. Define the general $i$-th commutator of $T$ with a measurable function $b_{i}$ by
$$[b_{i},T]_{i}(f_{1},\cdots,f_{m})(x):=b_{i}(x)T(f_{1},\cdots,f_{i},\cdots,f_{m})(x)-T(f_{1},\cdots,b_{i}f_{i},\cdots,f_{m})(x).$$
The linear commutator is defined by
$$[\Sigma  \vec{b},T](f_{1},\cdots,f_{m})(x):=\sum_{i=1}^{m}[b_{i},T]_{i}.$$
The iterated commutator is defined by
$$[\Pi\vec{b},T](f_{1},\cdots,f_{m})(x):=[b_{m},\cdots,[b_{1},T]_{1}]\cdots]_{m}(f_{1},\cdots,f_{m})(x).$$

\section{The general theorem}

After established the following John-Nirenberg inequality in \cite{WZTams}, we solved the open problem proposed in \cite{Chaf2016}. There are two constants $c_{1}$ and $c_{2}$ such that for any $\lambda>0$ and any cube $Q$,
$$\big|\{x\in Q: |f(x)-c_{Q}|>\lambda\}\big|\leq c_{1}\exp\big(\frac{-c_{2}\lambda}{\|f\|_{BMO^{p,*}}}\big)|Q|, \quad 0<p<\infty,$$
where $c_{Q}$ be the value which minimizes $\frac{1}{|Q|}\int_{Q}|f(x)-c|^{p}dx$. Therefore, for any $0<p<\infty$,
$$\|f\|_{BMO}\approx\|f\|_{BMO^{p,*}}:=\sup_{Q}\inf_{c}\Big(\frac{1}{|Q|}\int_{Q}|f(x)-c|^{p}dx\Big)^{1/p}.$$

In order to obtain the characterized results of weak-type estimate, we need establish some weak type characterizations of $BMO$ in terms of $A_{\infty}$ weights.
\begin{lemma}\label{BMO-lem1}
For any $\omega\in A_{\infty}$ and $0<p<\infty$, we have that
$$\|f\|_{BMO}\approx\|f\|_{BMO^{p,*}_{\omega}}:=\sup_{Q}\inf_{c}\Big(\frac{1}{\omega(Q)}\int_{Q}|f(x)-c|^{p}\omega(x)dx\Big)^{1/p}<\infty.$$
\end{lemma}
\begin{proof}
Since $\|f\|_{BMO^{p,*}_{\omega}}\leq \|f\|_{BMO^{p}_{\mu,\omega}}\approx\|f\|_{BMO}$ with $\mu\in A_{\infty}$, we need show that $\|f\|_{BMO}\lesssim \|f\|_{BMO^{p,*}_{\omega}}$ only.
By the property of $A_{\infty}$ weights, we conclude that for any $\omega\in A_{\infty}$, there exists a constant $N>0$ such that $\omega\in A_{N}$. Then
\begin{align*}
\int_{Q}|f(x)-c_{Q}|^{p/N}dx&=\int_{Q}|f(x)-c_{Q}|^{p/N}\omega(x)^{1/N}\cdot \omega(x)^{-1/N}dx\\
&\leq  \Big(\int_{Q}|f(x)-c_{Q}|^{p}\omega(x)dx\Big)^{1/N}\Big(\int_{Q}\omega(x)^{\frac{1}{1-N}}dx\Big)^{\frac{N-1}{N}}\\
&\leq  [\omega]_{A_{N}}\|f\|^{p/N}_{BMO^{p,*}_{\omega}}|Q|,
\end{align*}
where $c_{Q}$ be the value which minimizes $\frac{1}{|Q|}\int_{Q}|f(x)-c|^{p/N}dx$. From the equivalence of $BMO$ and $BMO^{p,*}$, $\|f\|_{BMO}\lesssim \|f\|_{BMO^{p,*}_{\omega}}$ follows from here.
\end{proof}
\begin{lemma}\label{BMO-lem2}
For any $\omega\in A_{\infty}$ and $0<p<\infty$, we have that
$$\|f\|_{BMO}\approx\|f\|_{BMO^{p}_{\omega,*}}:=\sup_{Q}\sup_{\lambda>0}\inf_{c}\frac{\lambda}{\omega(Q)^{1/p}}\omega\Big(\{x\in Q:|f(x)-c|>\lambda\}\Big)^{1/p}<\infty.$$
\end{lemma}
\begin{proof}
By the direct computation, we arrive at $\|\cdot\|_{BMO^{p}_{\omega,*}}\leq \|\cdot\|_{BMO^{p,*}_{\omega}}$. Then, we need only to prove that $\|\cdot\|_{BMO^{q,*}_{\omega}}\lesssim \|\cdot\|_{BMO^{p}_{\omega,*}}$ with $0<q<p<\infty$.

Let $f\in BMO^{p}_{\omega,*}$. Given a fixed cube $Q \subset \mathbb{R}^n$ and for any $\lambda>0$, let $c_{Q}$ be the value which minimizes $\frac{\lambda}{\omega(Q)^{1/p}}\omega\big(\{x\in Q:|f(x)-c_{Q}|>\lambda\}\big)$. Therefore,
$$\frac{1}{\omega(Q)^{1/p}}\Big(\lambda^{p}\omega\big\{x\in Q:|f(x)-c_{Q}|>\lambda\big\}\Big)^{1/p}\leq \|f\|_{BMO^{p}_{\omega,*}};$$
that is,
$$\omega\big\{x\in Q:|f(x)-c_{Q}|>\lambda\big\}\leq \|f\|^{p}_{BMO^{p}_{\omega,*}}\omega(Q)\lambda^{-p}.$$
It follows that
\begin{eqnarray*}
\int_{Q}|f(x)-c_{Q}|^{q}\omega(x)dx&=&q\int_{0}^{\infty}\lambda^{q-1}\omega\big\{x\in Q:|f(x)-c_{Q}|>\lambda\big\}d\lambda\\
&\leq&q\int_{0}^{N}\lambda^{q-1}\omega(Q)d\lambda+q\int_{N}^{\infty}\lambda^{q-1}\|f\|^{p}_{BMO^{p}_{\omega,*}}|Q|\lambda^{-p}d\lambda\\
&=&\omega(Q)N^{q}+\frac{q}{p-q}\|f\|^{p}_{BMO^{p}_{\omega,*}}\omega(Q)N^{q-p}.
\end{eqnarray*}
Choose
$$N=\|f\|_{BMO^{p}_{\omega,*}}\Big(\frac{q}{p-q}\Big)^{1/p},$$
which gives
$$\bigg(\frac{1}{\omega(Q)}\int_{Q}|f(y)-c_{Q}|^{q}\omega(x)dy\bigg)^{1/q}\leq 2\Big(\frac{q}{p-q}\Big)^{1/p}\|f\|_{BMO^{p}_{\omega,*}}.$$
Then
$$\|f\|_{BMO^{q}_{\mu,\omega}}\leq 2\Big(\frac{q_{1}}{p-q}\Big)^{1/p}\|f\|_{BMO^{p}_{\omega,*}}$$
and the lemma follows.
\end{proof}

Lerner et al. in \cite[Theorem 3.6]{LOPTT2009} showed that the multilinear $A_{\vec{P}}$ condition has the following interesting characterization in terms of the linear $A_{p}$ classes. This lemma plays an important role in this paper.
\begin{lemma}\label{weights-lem1}
Let $\vec{\omega}=(\omega_{1},\cdots,\omega_{m})$ and $1\leq p_{1},\cdots,p_{m}<\infty$. Then $\vec{\omega}\in A_{\vec{P}}$ if and only if
\begin{equation*}
    \left\{
   \begin{array}{ll}\vspace{1ex}
\omega_{i}^{1-p'_{i}}\in A_{mp'_{i}}, i=1,\cdots,m,\vspace{1ex}\\
\nu_{\vec{\omega}}\in A_{mp},
   \end{array}
 \right.
\end{equation*}
where the condition $\omega_{i}^{1-p'_{i}}\in A_{p'_{i}}$ in the case $p_{i}=1$ is understood as $\omega_{i}^{1/m}\in A_{1}$.
\end{lemma}

\subsection{The general commutator of multilinear Calder\'{o}n-Zygmund operators I}\label{the general commutators 1}
It has been an open question whether commutator of multilinear Calder\'{o}n-Zygmund operators can be used to characterize $BMO$? Chaffee \cite{Chaf2016} first considered this problem and obtained the following characterized theorem.

\medskip

{\bf Theorem A} (cf. \cite{Chaf2016})\quad Suppose that $K$ is a homogeneous function of degree $-mn$, and there exists a ball $\mathbb{B}\subset \mathbb{R}^{mn}$ such that $1/K$ can be expended to a Fourier series in $\mathbb{B}$. If $1<p_1,\cdots,p_m,{p}<\infty$, $\frac 1p=\frac 1{p_1}+\cdots+\frac 1{p_m}$,
$$b\in BMO\,\Longleftrightarrow\, [b,T]_i:\, L^{p_1}\times\cdots\times L^{p_m}\to L^p.$$

It was then also revisited by Li and Wick \cite{LW2018} using different techniques. In both results are considered only under the assumption $p>1$, since the proof required the use of H\"{o}lder's inequality with $p$ and $p'$. We remove the restriction of that $p>1$ in \cite{WZTams}.
\medskip

Later, the necessity of bounded multilinear commutator was extended by \cite{ChafCruz2018}(or \cite{GLW2020} for quasi-Banach space) in a very general structure. They considered all the function spaces $X$ with norm $\|\cdot\|_{X}$ satisfy the following basic assumptions.
\begin{enumerate}[(i)]
  \item $\|f\|_X=\||f|\|_X$;
  \item if $|f|\leqslant |g|$ a.e., then $\|f\|_X\leqslant \|g\|_X$;
  \item if $\{f_n\}$ is a sequence of $X$ such that $|f_n|$ increases to $|f|$ a.e., then
  $\|f_n\|_X$ increases to $\|f\|_X$;
  \item if $A$ is a bounded set of $\mathbb{R}^n$, then $\|\chi_A\|_X<\infty$.
\end{enumerate}

It is should be pointed that $\vec{\omega}\in A_{\vec{P}}$ does not imply $\omega_{k}\in L^{1}_{loc}$ for any $k$ (\cite[Remark 7.2]{LOPTT2009}), which showed that $\omega_{k}\in A_{p_{k}}$ may be not valid, and the condition $(\mathrm{iv})$ is not true in the weighted case. Thus, they proved the necessity theory for multilinear commutators with making stronger assumption that $\vec{\omega}\in A_{\vec{P}}$.

\medskip

{\bf Theorem B} (cf. \cite{ChafCruz2018}).\quad{\it Given $\vec{P}$ with $p>1$, suppose $\omega_{k}\in A_{p_{k}}, k=1,\cdots,m,$ and $\vec{\omega}\in A_{\vec{P}}$. If $T$ is a regular bilinear singular integral, and $b$ is function such that for $i=1,\cdots,m$, $[b,T]_{i}: L^{p_{1}}\times\cdots\times L^{p_{m}}(\omega_{m})\rightarrow L^{p}(\nu_{\vec{\omega}})$, then $b\in BMO$.}

\medskip

Guo, Lian and Wu \cite{GLW2020} removed the restriction of that $p>1$ and gave a characterized result of weak-type endpoint estimate. However, the weights are still restricted in a narrower class. Later, the results with the genuinely multilinear weights had been resolved in \cite{L2020} for $1<p_{1},\cdots,p_{m}<\infty$. Then we consider the endpoint case in this subsection.

\begin{theorem}\label{main1}
Let $m\in \mathbb{N}$ and $T$ be a $m$-linear Calder\'{o}n-Zygmund operator with kernel $K$ and there exists a ball $\mathbb{B}\subset \mathbb{R}^{mn}$ such that $1/K$ can be expended to a Fourier series in $\mathbb{B}$.
Then the following three statements are equivalent:
\begin{enumerate}
\item [\rm(1)] $b\in BMO$.
\item [\rm(2)] For $\nu_{\vec{\omega}}=\prod_{i=1}^{m}\omega_{i}^{1/m}$ and $\vec{\omega}\in A_{1,\cdots,1}$, $\Phi(t)=t(1+\log^{+}t)$,
$$\nu_{\vec{\omega}}\big(\{x\in \mathbb{R}^{n}:|[b,T]_{i}(f_{1},\cdots,f_{m})|>\lambda^{m}\}\big)\lesssim \prod_{i=1}^{m}\Big(\int_{\mathbb{R}^n}\Phi(\frac{|f(y)|}{\lambda})\omega_{i}(y)dy\Big)^{1/m}$$
for any $\lambda>0$.
\end{enumerate}
\end{theorem}
\begin{proof}
For simplicity, we only present the proof for $[b,T]_{1}$, since the other cases can be
treated similarly. We need to prove $(2)\Longrightarrow (1)$, since $(1)\Longrightarrow (2)$ follow directly from \cite[Theorems 3.16]{LOPTT2009}.

Let $z_{0}\in \mathbb{R}^n$ such that $|(z_{0},\cdots,z_{0})|>m\sqrt{n}$ and let $\delta\in (0,1)$ small enough. Take $\mathbb{B}=B\big((z_{0},\cdots,z_{0}),\delta\sqrt{mn}\big)\subset \mathbb{R}^{mn}$ be the ball, these conditions guarantee that $\mathbb{B}\cap \{0\}=\emptyset$, avoiding any potential singularity of $1/K$. Which we can express $1/K$ as an absolutely convergent Fourier series of the form
$$\frac{1}{K(y_{1},\cdots,y_{m})}=\sum_{l}a_{l}e^{\mathrm{i}v_{l}\cdot(y_{1},\cdots,y_{m})}, \quad (y_{1},\cdots,y_{m})\in \mathbb{B},$$
with $\sum_{l}|a_{l}|<\infty$ and we do not care about the vectors $v_{l}\in \mathbb{R}^{mn},$ but we will at times express them as $v_{l}=(v_{l}^{1},\cdots,v_{l}^{m})\in \mathbb{R}^{mn}.$
Set $z_{1}=\delta^{-1}z_{0}$ and note that
$$\big(|y_{1}-z_{1}|^{2}+\cdots+|y_{m}-z_{1}|^{2}\big)^{1/2}<\sqrt{mn}\Rightarrow \big(|\delta y_{1}-z_{0}|^{2}+\cdots+|\delta y_{m}-z_{0}|^{2}\big)^{1/2}<\delta \sqrt{mn}.$$
Then for any $(y_{1},\cdots,y_{m})$ satisfying the inequality on the left, we have
$$\frac{1}{K(y_{1},\cdots,y_{m})}=\frac{\delta^{-mn}}{K(\delta y_{1},\cdots,\delta y_{m})}=\delta^{-mn}\sum_{j}a_{j}e^{\mathrm{i}\delta v_{j}\cdot(y_{1},\cdots,y_{m})}.$$
Let $Q=Q(x_{0},r)$ be any arbitrary cube in $\mathbb{R}^n$. Set $\tilde{z}=x_{0}+rz_{1}$ and take $Q'=Q(\tilde{z},r)\subset \mathbb{R}^n$. So for any $x\in Q$ and $y_{1},\cdots,y_{m}\in Q'$, we have for any $i=1,2,\cdots,m$,
$$\Big|\frac{x-y_{i}}{r}-z_{1}\Big|\leq \Big|\frac{x-x_{0}}{r}\Big|+\Big|\frac{y_{i}-\tilde{z}}{r}\Big|\leq \sqrt{n},$$
which implies that
$$\bigg(\Big|\frac{x-y_{1}}{r}-z_{1}\Big|^{2}+\cdots+\Big|\frac{x-y_{m}}{r}-z_{1}\Big|^{2}\bigg)^{1/2}\leq \sqrt{mn}.$$
This yields that there is a constant $C_{m,n}$ such that $Q'\subset C_{m,n}Q$ and $Q\subset C_{m,n}Q'$. The doubling property of $A_{\infty}$ weight deduce that $\mu(Q')\approx \mu(Q)$ for any $\mu\in A_{\infty}$.

We can now estimate as follows. For any $k=1,\cdots,m$, we write
$\mu_{k}:=\omega_{k}^{-1},$
and let $s(x)=\overline{\mathrm{sgn}(\int_{Q'}(b(x)-b(y))\mu_{1}(y)dy)}$. Then
\begin{eqnarray*}
\begin{aligned}
&|b(x)-b_{\mu_{1},Q'}|
=\frac{s(x)}{\Pi_{k=1}^{m}\mu_{k}(Q')}\int_{(Q')^{m}}\big(b(x)-b(y_{1})\big)\prod_{k=1}^{m}\mu_{k}(y_{k}) d\vec{y}.
\end{aligned}
\end{eqnarray*}
Define the functions
\begin{equation*}
    \left\{
   \begin{array}{ll}\vspace{1ex}
g^{k}_{l}(y_{k})=e^{-\mathrm{i}\frac{\delta}{r}v^{k}_{l}\cdot y_{k}}\mu_{k}(y_{k})\chi_{Q'}(y_{k})\frac{|Q'|}{\mu_{k}(Q')}, & k=1,\cdots,m,\\
h_{l}(x)=e^{\mathrm{i}\frac{\delta}{r}v_{l}\cdot (x,\cdots,x)}\chi_{Q}(x)s(x),
   \end{array}
 \right.
\end{equation*}
which shows that
\begin{align*}
|b(x)-b_{\mu_{1},Q'}|&= s(x)\frac{r^{mn}\delta^{-mn}}{\Pi_{k=1}^{m}\mu_{k}(Q')}
\int_{(Q')^{m}}(b(x)-b(y_{1}))K(x-y_{1},\cdots,x-y_{m})\\
&\qquad\times\sum_{l}a_{l}e^{i\frac{\delta}{r}v_{l}\cdot(x-y_{1},\cdots,x-y_{m})}\prod_{k=1}^{m}\mu_{k}(y_{k})d\vec{y}\\
&=\delta^{-mn}\sum_{l}a_{l}[b,T]_{1}(g^{1}_{l},\cdots,g^{m}_{l})(x)h_{l}(x).
\end{align*}
It follows from Lemma \ref{weights-lem1} that $\omega^{1/m}_{k}\in A_{1}$ for all $k=1,\cdots,m$, then
$$\|\mu_{k}\chi_{Q'}\|_{L^{\infty}}=\|\omega_{k}^{-1/m}\chi_{Q'}\|^{m}_{L^{\infty}}\lesssim \big(\frac{|Q'|}{\omega^{1/m}_{k}(Q')}\big)^{m}. $$
On the other hand,
\begin{align*}
|Q|^{2}&=\Big(\int_{Q}\omega_{k}(x)^{\frac{1}{2m}}\omega_{k}(x)^{-\frac{1}{2m}}dx\Big)^{2}\\
&\leq \int_{Q}\omega_{k}(x)^{1/m}dx\int_{Q}\omega_{k}(x)^{-1/m}dx\\
&\leq \int_{Q}\omega_{k}(x)^{1/m}dx\big(\int_{Q}\omega_{k}(x)^{-1}dx\big)^{1/m}|Q|^{\frac{m-1}{m}},
\end{align*}
which gives us that
$$\frac{|Q'|}{\mu_{k}(Q')}=\frac{|Q'|}{\int_{Q'}\omega_{k}(x)^{-1}dx}\lesssim \big(\frac{\omega^{1/m}_{k}(Q')}{|Q'|}\big)^{m}.$$
Combining the above computations we obtain $\|g^{k}_{l}\|_{L^{\infty}}\lesssim 1$.

If $0<\lambda<\max\{\|g^{1}_{l}\|_{L^{\infty}},\cdots, \|g^{m}_{l}\|_{L^{\infty}}\}\big(\sum_{l}|a_{l}|\big)^{1/m}$, it is easy to see that
\begin{align*}
\inf_{c}\frac{\lambda}{\nu_{\vec{\omega}}(Q)}\nu_{\vec{\omega}}\big(x\in Q:|b(x)-c|>\lambda^{m}\big)\lesssim \lambda\lesssim \big(\sum_{l}|a_{l}|\big)^{1/m}.
\end{align*}
If $\max\{\|g^{1}_{l}\|_{L^{\infty}},\cdots, \|g^{m}_{l}\|_{L^{\infty}}\}\big(\sum_{l}|a_{l}|\big)^{1/m} \leq \lambda<\infty$, then $$\Phi\Big(\frac{\big(\sum_{l}|a_{l}|\big)^{1/m}|g^{k}_{l}(y_{k})|}{\lambda}\Big)=\frac{|g^{k}_{l}(y_{k})|\big(\sum_{l}|a_{l}|\big)^{1/m}}{\lambda}.$$
We can now continue the above estimate as follows:
\begin{align*}
\inf_{c}\nu_{\vec{\omega}}\big(x\in Q:|b(x)-c|>\lambda^{m}\big)
&\leq \nu_{\vec{\omega}}\big(x\in Q:|b(x)-b_{\mu_{1},Q'}|>\lambda^{m}\big)\\
&\lesssim \nu_{\vec{\omega}}\big(x\in Q:\sum_{l}|a_{l}|\big|[b,T]_{1}(g^{1}_{l},\cdots,g^{m}_{l})(x)\big|>\lambda^{m}\big)\\
&\lesssim \prod_{k=1}^{m}\Big(\int_{Q'}\Phi\Big(\frac{|g^{k}_{l}(y_{k})|\big(\sum_{l}|a_{l}|\big)^{1/m}}{\lambda}\Big)\omega_{k}(y_{k})dy_{k}\Big)^{1/m}\\
&\lesssim \frac{\big(\sum_{l}|a_{l}|\big)^{1/m}}{\lambda}\prod_{k=1}^{m}\Big(\frac{|Q'|^{2}}{\mu_{k}(Q')}\Big)^{1/m}.
\end{align*}
Since $\omega^{1/m}_{1},\cdots,\omega^{1/m}_{m}\in A_{1}$, we have $\frac{\omega^{1/m}_{k}(Q')}{|Q'|}\lesssim\inf_{x\in Q'}\omega_{k}(x)^{1/m}, k=1,\cdots,m$. Then
\begin{align*}
\omega^{1/m}_{1}(Q')\cdots\omega^{1/m}_{m}(Q')&\lesssim |Q'|^{m}\inf_{x\in Q}\omega_{1}(x)^{1/m}\cdots \inf_{x\in Q'}\omega_{m}(x)^{1/m}\\
&\lesssim |Q'|^{m-1}\int_{Q'}\omega_{1}(x)^{1/m}\cdots \omega_{m}(x)^{1/m}dx=|Q'|^{m-1}\nu_{\vec{\omega}}(Q').
\end{align*}
Combining with the above estimates, we conclude that
\begin{align*}
\inf_{c}\frac{\lambda}{\nu_{\vec{\omega}}(Q)}\nu_{\vec{\omega}}\big(x\in Q:|b(x)-c|>\lambda^{m}\big)\lesssim \big(\sum_{l}|a_{l}|\big)^{1/m}.
\end{align*}
From Lemma \ref{BMO-lem2}, we have that $b\in BMO$ and the proof is complete.
\end{proof}

\subsection{The general commutator of multilinear Calder\'{o}n-Zygmund operators II}\label{the general commutators 2}

The investigation on boundedness of the commutator for Calder\'{o}n-Zygmund operators with certain rough keneral are usually valid. The
necessity theory for bounded commutators only for a very nice subclass of Calder\'{o}n-Zygmund operators, since most of the previous proofs is expand the kernel locally by Fourier series. In 2018, Guo, Lian and Wu \cite{GLW2020} gave unified criterions on the necessity of commutators of non-smooth Calder\'{o}n-Zygmund operators.

\medskip

{\bf Theorem C.}(cf. \cite{GLW2020}) Let $T$ be a $m$-linear Calder\'{o}n-Zygmund operator with kernel $K$ satisfying following local properties: there exists an open cone $\widetilde{\Gamma}$ of $(\mathbb{R}^{n})^m$
  whose vertex is $0$, such that :
  \begin{enumerate}[(1)]
  \item[(i)] lower and upper bound: $$\frac{c}{(\sum_{k=1}^m|x-y_k|)^{mn}}\leqslant K(x,y_1,\cdots,y_m) \leqslant \frac{C}{(\sum_{k=1}^m|x-y_k|)^{mn}}$$
  for all $(x-y_1,\cdots,x-y_m)\in \widetilde{\Gamma}$ with $0<c<C$, or $c<C<0$;
  \item[(ii)]
  for any open cone $\overline{\Gamma}\subset \widetilde{\Gamma}$,
  there exists a sequence $\{h_l=(h_l^1,\cdots,h_l^m)\}_{l=1}^{\infty}$
  satisfying that $h_l\in \overline{\Gamma}$, $|h_l|\rightarrow \infty$ as $l\rightarrow \infty$, and
  \begin{equation*}
    \begin{aligned}
      &\Big\|\int_{\prod\limits_{{k=1}}^m(Q-\sqrt[n]{Q}h_l^k)}\big|K(\cdot,y_1,\cdots,y_m) \\
      &\qquad
      -K(\cdot,y_1,\cdots,y_{i-1},a_{Q}-\sqrt[n]{|Q|}h_l^i,\cdots,y_m)\big|dy_1\cdots\,dy_m \Big\|_{L^{\infty}(Q)}
      \times \frac{|h_l|^{mn}}{|Q|}\rightarrow 0,\,
    \end{aligned}
  \end{equation*}
as $l \rightarrow \infty$ uniformly for all cubes Q, where $a_Q$ denotes the center of $Q.$
  \end{enumerate}

\medskip

Let $b\in L_{loc}^1(\mathbb{R}^n)$ and $[b,T]_i$ be the $i$-th commutator generated by $T$ with $b$, $1\le i\le m$. Then the following three statements are equivalent:
\begin{enumerate}
  \item [\rm(1)] $b\in BMO(\mathbb{R}^n)$.

 \item  [\rm(2)] For $1<p_k<\infty$, $\omega_k\in A_{p_k}$, $k=1,\cdots, m$, with $1/p=1/p_1+\cdots+1/p_m$,
 $$\|[b,T]_i(f_1,\cdots,f_m)\|_{L^p(\nu_{\vec{\omega}})}\lesssim \prod_{k=1}^m\|f_k\|_{L^{p_k}(\omega_k)}.$$

 \item [\rm(3)] For $\omega_k\in A_1$, $k=1,\cdots,m$, $\Phi(t)=t(1+\log^+t)$,
  $$\nu_{\vec{\omega}}(\{x\in \mathbb{R}^n: |[b,T]_i(f_{1},\cdots,f_{m})(x)|>\lambda^m\})\lesssim \prod_{k=1}^m\left(\int_{\mathbb{R}^n}\Phi(\frac{|f_{k}(y)|}{\lambda})\omega_k(y)dy\right)^{{1}/{m}}$$
  for any $\lambda>0$.

\end{enumerate}

Our next theorem will remove the restriction of that $\omega_{k}\in A_{p_{k}}$ for all $k=1,\cdots,m$ and modify the regularity condition of the kernel.

\begin{theorem}\label{GC2}
Let $T$ be a $m$-linear Calder\'{o}n-Zygmund operator with kernel $K$ satisfying following local properties: there exists an open cone $\widetilde{\Gamma}$ of $(\mathbb{R}^{n})^m$
whose vertex is $0$ such that for any $\mu_{1},\cdots,\mu_{m}\in A_{\infty}$,
  \begin{enumerate}[(1)]
  \item[(i)] lower and upper bound: $$\frac{c}{(\sum_{k=1}^m|x-y_k|)^{mn}}\leqslant K(x,y_1,\cdots,y_m) \leqslant \frac{C}{(\sum_{k=1}^m|x-y_k|)^{mn}}$$
  for all $(x-y_1,\cdots,x-y_m)\in \widetilde{\Gamma}$ with $0<c<C$, or $c<C<0$;
  \item[(ii)]
  for any open cone $\overline{\Gamma}\subset \widetilde{\Gamma}$,
  there exists a sequence $\{h_l=(h_l^1,\cdots,h_l^m)\}_{l=1}^{\infty}$
  satisfying that $h_l\in \overline{\Gamma}$, $|h_l|\rightarrow \infty$ as $l\rightarrow \infty$, and
  \begin{equation*}
    \begin{aligned}
|h_l|^{mn} \times &\Big\|\int_{\prod\limits_{{k=1}}^m(Q-\sqrt[n]{Q}h_l^k)}\big|K(\cdot,y_1,\cdots,y_m) \\
      &\qquad
      -K(\cdot,y_1,\cdots,y_{i-1},a_{Q}-\sqrt[n]{|Q|}h_l^i,\cdots,y_m)\big|\prod_{k=1}^{\infty}\mu_{k}(y_{k})d\vec{y} \Big\|_{L^{\infty}(Q)}
      \rightarrow 0,\,
    \end{aligned}
  \end{equation*}
as $l \rightarrow \infty$ uniformly for all cubes Q, where $a_Q$ denotes the center of $Q.$
  \end{enumerate}

\medskip

Let $b\in L_{loc}^1(\mathbb{R}^n)$ and $[b,T]_i$ be the $i$-th commutator generated by $T$ with $b$, $1\le i\le m$. Then the following three statements are equivalent:
\begin{enumerate}
  \item [\rm(1)] $b\in BMO(\mathbb{R}^n)$.

 \item [\rm(2)] For $1<p_1,\cdots,p_{m}<\infty$ and $\vec{\omega}\in A_{\vec{P}}$ with $1/p=1/p_1+\cdots+1/p_m$, $$\|[b,T]_i(f_1,\cdots,f_m)\|_{L^p(\nu_{\vec{\omega}})}\lesssim \prod_{k=1}^m\|f_k\|_{L^{p_k}(\omega_k)}.$$

 \item [\rm(3)] For $1<p_1,\cdots,p_{m}<\infty$ and $\vec{\omega}\in A_{\vec{P}}$ with $1/p=1/p_1+\cdots+1/p_m$, $$\|[b,T]_i(f_1,\cdots,f_m)\|_{L^{p,\infty}(\nu_{\vec{\omega}})}\lesssim \prod_{k=1}^m\|f_k\|_{L^{p_k}(\omega_k)}.$$

 \item [\rm(4)] For $\vec{\omega}\in A_{1,\cdots,1}$, $\Phi(t)=t(1+\log^+t)$,
  $$\nu_{\vec{\omega}}(\{x\in \mathbb{R}^n: |[b,T]_i(f)(x)|>\lambda^m\})\lesssim \prod_{k=1}^m\left(\int_{\mathbb{R}^n}\Phi(\frac{|f_{k}(y)|}{\lambda})\omega_k(y)dy\right)^{{1}/{m}}$$
  for any $\lambda>0$.

\end{enumerate}
\end{theorem}

\begin{proof}
Since the implications $(1) \Longrightarrow(2)\Longrightarrow (3)$ and $(1)\Longrightarrow(4)$ follow readily, we
only have to prove $(3) \Longrightarrow(1)$ and $(4)\Longrightarrow(1)$. Without loss of generality, we only deal with the kernel is positive and the case $i=1$. Choose a constant $\tau_0$ and a nonempty open cone
$\Gamma\subset \widetilde{\Gamma}\subset \mathbb{R}^{mn}$ with vertex at the origin,
such that for any $u\in Q_0^m:=[-1,1]^n\times\cdots\times [-1,1]^{n}$, $v\in \Gamma_{\tau_0}:=\Gamma\cap \mathbb{B}^c(0,\tau_0)$, we have $u+v\in \widetilde{\Gamma}$.

We first deal with the case $1<p_{1},\cdots,p_{m}<\infty$. Let $\mu_{k}=\omega_{k}^{1-p'_{k}}\in A_{\infty}, k=1,\cdots,m$. We can find a sequence $\{h_l=(h_l^1,\cdots,h_l^m)\}_{l=1}^{\infty}$ satisfying that $h_l\in \Gamma_{\tau_0}$, and $|h_l|\rightarrow \infty$
as $l\rightarrow \infty$. For any cubes $Q_{l}:=Q(a_{Q}-\rho h^{1}_{l},\rho)$ and $Q_{k,l}=a_{Q}+\rho h^{1}_{l}-\rho h^{k}_{l}$ with $k=1,\cdots,m,$ we have
\begin{equation*}
  \begin{aligned}
 |h_l|^{mn}\left\|\int_{Q_{1,l}\times\cdots\times Q_{m,l}}
  |K(\cdot,y_{1},\cdots,y_{m})-K(\cdot,y_{1},\cdots,y_{m})\big|\prod_{k=1}^{m}\mu_{k}(y_{k})d\vec{y}\right\|_{L^{\infty}(Q_{l})}\rightarrow 0
  \end{aligned}
\end{equation*}
as $l\rightarrow \infty$, uniformly for all cubes $Q$.

For a fixed cube $Q(a_{Q},\rho)=:Q_{1,l}$, we may assume that
\begin{equation*}
  \int_{Q_{1,l}}b(y_{1})\mu_{1}(y_{1})dy_{1}=0.
\end{equation*}
Otherwise, we use $b(y_{1})-b_{\mu_{1},Q_{1,l}}$ instead of $b(y_{1})$. Take
\begin{equation*}
    \left\{
   \begin{array}{ll}\vspace{1ex}
\phi_1(x)=\left(sgn(b(x))\mu_{1}(x)-\frac{1}{\mu_{1}(Q_{1,1})}\int_{Q_{1,l}}sgn(b(y))\mu_{1}(y)dy\right)\chi_{Q_{1,l}}(x),\\
\phi_j:=\mu_{j}\chi_{Q_{j,l}}, \qquad j=2,\cdots,m.
   \end{array}
 \right.
\end{equation*}
Then $-2\mu_{1}\chi_{Q_1}\leq \phi_1\leq 2\mu_{1}\chi_{Q_1}$, $b\phi_1\geqslant 0$.

For $x\in Q_{l}$, $y_{k}\in Q_{k,l}$ with $k=1,\cdots,m$, we have
 $$\frac{(x-y_{1},\cdots,x-y_{m})}{\rho}\in Q_0^m+h_l\subset Q_0^m+\Gamma_{\tau_0}\subset \widetilde{\Gamma}.$$
It follows that
\begin{equation}\label{kenerl}
  K(x,y_{1},\cdots,y_{m})\sim \frac{1}{(|x-y_{1}|+\cdots+|x-y_{m}|)^{mn}}\sim \frac{1}{(\rho|h_l|)^{mn}}.
\end{equation}

Recalling $b\phi_1\geqslant 0$, we obtain that
\begin{equation}\label{A1}
  \begin{aligned}
    |T(b\phi_1,\cdots,\phi_{m})(x)|
    &=
    \int_{ Q_{1,l}\times \cdots\times Q_{m,l}} K(x,y_{1},\cdots,y_{m})b(y_{1})\prod_{k=1}^{m}\phi_k(y_{k})d\vec{y}
    \\
    &\gtrsim
    \frac{\Pi_{k=2}^{m}\mu_{k}(Q_{k,l})}{(\rho|h_l|)^{mn}}\int_{Q_1} b(y_{1})\mu_1(y_{1})dy_{1}
    =
    \frac{\Pi_{k=1}^{m}\mu_{k}(Q_{k,l})}{(\rho|h_l|)^{mn}}B,
  \end{aligned}
\end{equation}
where
$$B:=\frac{1}{\mu_{1}(Q_{1,l})}\int_{Q_{1,l}}|b(y)|\mu_{1}(y)dy.$$
By the direct calculation, we arrive at
\begin{equation}\label{A2}
  \begin{aligned}
  \big|b(x)T(\phi_1,\cdots,\phi_m)(x)\big|
    &\lesssim
   |b(x)|\|T(\phi_1,\cdots,\phi_m)\|_{L^{\infty}(Q_{l})},\quad\forall\, x\in Q_{l}.
  \end{aligned}
\end{equation}
The combination of (\ref{A1}) and (\ref{A2}) then yields that
\begin{equation}\label{A1A2}
    \begin{aligned}
  &\big|[b,T]_1(\phi_1,\cdots,\phi_m)(x)\big|
  \geqslant
  A_1\frac{\Pi_{k=1}^{m}\mu_{k}(Q_{k,l})B}{(\rho|h_l|)^{mn}}
  -A_2|b(x)|\|T(\phi_1,\cdots,\phi_m)\|_{L^{\infty}(Q_{l})}.
    \end{aligned}
\end{equation}
On the other hand, take $$ \psi_k=\mu_{k}\chi_{Q_{k,l}},k=1,\cdots,m.$$
We have
\begin{equation}\label{A3A4}
  \begin{aligned}
    &\big|[b,T]_1(\psi_1,\cdots,\psi_m)(x)\big|\\
    &\geqslant
    \big|[b,T]_1(\psi_1,\cdots,\psi_m)(x)\chi_{Q_{l}}(x)\big| \\
    &\geqslant
    \big|b(x)T(\psi_1,\cdots,\psi_m)(x)\chi_{Q_{l}}(x)\big|-\big|T(b\psi_1,\psi_2)(x)\chi_{Q_{l}}(x)\big|.
  \end{aligned}
\end{equation}
Recalling the lower and upper bound of keneral $K$ in (\ref{kenerl}),
\begin{equation}\label{A3}
  \begin{aligned}
   & |T(b\psi_1,\cdots,\psi_m)(x)|\\
   & \leqslant
    \int_{Q_{1,l}\times \cdots\times Q_{m,l}} |K(x,y_{1},\cdots,y_{m})||b(y_{1})|\prod_{k=1}^{m}\psi(y_{k})d\vec{y}
    \\
   & \lesssim
    \frac{\Pi_{k=2}^{m}\mu_{k}(Q_{k,l})}{(\rho|h_l|)^{mn}}\int_{Q_{1,l}} |b(y_{1})|\mu(y_{1})dy_{1}
    =
    \frac{\Pi_{k=1}^{m}\mu_{k}(Q_{k,l})B}{(\rho|h_l|)^{mn}},\quad\forall\, x\in Q_{l}.
  \end{aligned}
\end{equation}
Also, for $x\in Q_{l}$,
\begin{equation}\label{A4}
  \begin{aligned}
    |b(x)T(\psi_1,\cdots,\psi_m)(x)|
    =&
    \left|b(x)\int_{Q_{1,l}\times\cdots\times Q_{m,l}}K(x,y_{1},\cdots,y_{m})d\vec{y}\right|
    \\
    \gtrsim &
    \frac{|b(x)|\cdot \rho^{mn}}{(\rho|h_l|)^{mn}}=\frac{|b(x)|}{|h_l|^{mn}}.
  \end{aligned}
\end{equation}
The combination of (\ref{A3A4}), (\ref{A3}) and (\ref{A4}) yields that
\begin{equation}\label{A3A4,2}
  \big|[b,T]_1(\psi_1,\cdots,\psi_m)(X)\big|
  \geqslant
   A_{3}\frac{|b(x)|}{|h_l|^{mn}}-A_{4}\frac{\Pi_{k=1}^{m}\mu_{k}(Q_{k,l})B}{(\rho|h_l|)^{mn}}.
\end{equation}
Denote
\begin{equation*}
  \Xi =\frac{A_{2}|h_l|^{mn}\|T(\phi_1,\cdots,\phi_m)\|_{L^{\infty}(Q_{l})}}{A_{3}}.
\end{equation*}
Using (\ref{A1A2}),(\ref{A3A4,2}), and the boundedness of $[b,T]_{1}$, we obtain that
\begin{equation}\label{Xi}
  \begin{aligned}
   & (1+\Xi)\big|[b,T]_1(\phi_{1},\cdots,\phi_{m})(x)\big|+\big|[b,T]_1(\psi_{1},\cdots,\psi_{m})(x)\big|
    \\
   & \geqslant
    (A_1-\Xi A_4)\frac{\Pi_{k=1}^{m}\mu_{k}(Q_{k,l})B}{(\rho|h_l|)^{mn}}.
  \end{aligned}
\end{equation}

Now we check that $\Xi$ can be chosen small for sufficient large $l$.
For $x\in Q_{l}$, it follows from the definitions of $\phi_{i}$ that
\begin{equation*}
  \begin{aligned}
  &  |T(\phi_1,\cdots,\phi_m)(x)|
    =
    \Big|\int_{Q_{1,l}\times\cdots \times Q_{m,l}}K(x,y_{1},\cdots,y_{m})\prod_{i=k}^{m}\phi_1(y_{1})d\vec{y}\Big|
    \\
   & \leqslant
    2\int_{ Q_{1,l}\times\cdots \times Q_{m,l}}\left|K(x,y_{1},\cdots,y_{m})-K(x,a_{Q},y_{2},\cdots,y_{m})\right|\prod_{k=1}^{m}\mu_{i}(y_{i})d\vec{y}.
  \end{aligned}
\end{equation*}
Then,
\begin{equation*}
  \begin{aligned}
    &\|T(\phi_1,\cdots,\phi_m)\|_{L^{\infty}(Q_{l})}\\
    &\lesssim
    \left\|\int_{Q_{1,l}\times\cdots \times Q_{m,l}}\left|K(x,y_{1},\cdots,y_{m})-K(x,a_{Q},y_{2},\cdots,y_{m})\right|\prod_{k=1}^{m}\mu_{k}(y_{k})d\vec{y}\right\|_{L^{\infty}(Q_{l})}
  \end{aligned}
\end{equation*}
This shows that $ \Xi \rightarrow 0$ as $l \rightarrow \infty$. We can make $\Xi\leqslant \min\{A_1/(2A_4),1\}$ and $A_1-\Xi A_4\geqslant \frac{A_1}{2}$, \eqref{Xi} implies that
\begin{equation*}
 \frac{\Pi_{k=1}^{m}\mu_{k}(Q_{l})}{|Q_{l}|^{m}}B\leq A_{5} \Big(\big|[b,T]_1(\phi_{1},\cdots,\phi_{m})(x)\big|+\big|[b,T]_1(\psi_{1},\cdots,\psi_{m})(x)\big|\Big).
\end{equation*}
Recall the side length of $Q_{k,1}$ and $Q_{l}$ are $\rho$.
So, there exists $s$ depend only on $h_{l}$,
such that $Q_{k,l}\subset s Q_{l}$, $k=1,\cdots,m$.
We have now completed this proof.

Taking $\lambda=\frac{\Pi_{k=1}^{m}\mu_{k}(Q_{l})}{|Q_{l}|^{m}}\frac{A_{5}^{-1}B}{4}$, we get
  \begin{equation*}
  \begin{aligned}
 \nu_{\vec{\omega}}(Q_{l})= &  \nu_{\vec{\omega}}(\{x\in Q_{l}: \frac{\Pi_{k=1}^{m}\mu_{k}(Q_{l})}{|Q_{l}|^{m}}A_{5}^{-1}B>2\lambda\})\\
     \leqslant  &  \nu_{\vec{\omega}}(\{x\in Q_{l}: |[b,T]_1(\phi_1,\cdots,\phi_m)(x)|+|[b,T]_1(\psi_1,\cdots,\psi_m)(x)|>2\lambda\})
    \\
  \leqslant &
    \nu_{\vec{\omega}}(\{x\in Q_{l}: |[b,T]_1(\phi_1,\cdots,\phi_m)(x)|>\lambda\})\\
   & +
    \nu_{\vec{\omega}}(\{x\in Q_{l}: |[b,T]_1(\psi_1,\cdots,\psi_m)(x)|>\lambda\}).
  \end{aligned}
  \end{equation*}
By the boundedness of $[b,T]_{1}$ from $L^{p_{1}}(\omega_{1})\times \cdots\times L^{p_{m}}(\omega_{m})$ to $L^{p,\infty}(\nu_{\vec{\omega}})$, we deduce that
  \begin{equation*}
    \begin{aligned}
     & \lambda\nu_{\vec{\omega}}(\{x\in Q_{l}: |[b,T]_1(\phi_1,\phi_2)(x)|>\lambda\})^{1/p}
     \lesssim
      \prod_{k=1}^{m}\|\phi_{k}\|_{L^{p_{k}}(\omega_{k})}
  \lesssim \prod_{k=1}^{m} \mu_{k}(Q_{l})^{1/p_{k}},
    \end{aligned}
  \end{equation*}
sine $Q_{k,l}\subset s Q_{l}$ and the doubling property of $\mu_k$. Similarly,
  \begin{equation*}
     \lambda\nu_{\vec{\omega}}(\{x\in Q_{l}: |[b,T]_1(\psi_1,\cdots,\psi_m)(x)|>\lambda\})^{1/p}
      \lesssim
      \prod_{k=1}^{m} \mu_{k}(Q_{l})^{1/p_{k}}.
  \end{equation*}
Combining with the above estimates, we conclude that
\begin{equation*}
  \begin{aligned}
    B&\lesssim \frac{\lambda|Q_{l}|^{m}}{\Pi_{k=1}^{m}\mu_{k}(Q_{l})}
    \lesssim \prod_{k=1}^{m} \mu_{k}(Q_{l})^{-1/p'_{k}} \nu_{\vec{\omega}}(Q_{l})^{-1/p}|Q_{l}|^m.
  \end{aligned}
\end{equation*}
We use the fact
\begin{equation*}
  \begin{aligned}
|Q_{l}|^{m}&=\Big(\int_{Q_{l}}\nu_{\vec{\omega}}(x)^{\frac{1}{mp}}\prod_{k=1}^{m}\omega_{k}(x)^{-\frac{1}{mp_{k}}}dx\Big)^{m}\leq \nu_{\vec{\omega}}(Q_{l})^{1/p}\prod_{k=1}^{m}\mu_{k}(Q_{l})^{1/p'_{k}}
  \end{aligned}
\end{equation*}
to deduce that $B\lesssim 1$. Then, we have now completed the proof of $(3)\Longrightarrow (1)$.

Next, we deal with the endpoint case. Let $\mu_{k}:=\omega_{k}^{-1}, k=1,\cdots,m,$ and
\begin{equation*}
    \left\{
   \begin{array}{ll}\vspace{1ex}
\tilde{\phi}_1(x)=\left(sgn(b(x))\mu_{1}(x)-\frac{1}{\mu_{1}(Q_{1,1})}\int_{Q_{1,l}}sgn(b(y))\mu_{1}(y)dy\right)\Big(\frac{\omega_{1}^{1/m}(Q_{1,l})}{|Q_{1,l}|}\Big)^{m}\chi_{Q_1}(x),\\
\tilde{\phi}_j(x):=\Big(\frac{\omega_{j}^{1/m}(Q_{j,l})}{|Q_{j,l}|}\Big)^{m}\mu_{j}(x)\chi_{Q_{j,l}}(x), \qquad \qquad j=2,\cdots,m,\\
\tilde{\psi}_{k}(x):=\Big(\frac{\omega_{k}^{1/m}(Q_{k,l})}{|Q_{k,l}|}\Big)^{m}\mu_{k}(x)\chi_{Q_{k,l}}(x), \qquad \qquad k=1,\cdots,m.
   \end{array}
 \right.
\end{equation*}
Then $\|\tilde{\phi}_{k}\|_{L^{\infty}}, \|\tilde{\psi}_{k}\|_{L^{\infty}}\lesssim 1$ and $\|\tilde{\phi}_{k}\|_{L^{1}(\omega_{k})}=\|\tilde{\psi}_{k}\|_{L^{1}(\omega_{k})}=|Q_{l}|, k=1,\cdots,m.$
The same arguments as in the proof above, we obtain
\begin{equation*}
B\leq A_{6}\Big( \big|[b,T]_1(\tilde{\phi}_{1},\cdots,\tilde{\phi}_{m})(x)\big|+\big|[b,T]_1(\tilde{\psi}_{1},\cdots,\tilde{\psi}_{m})(x)\big|\Big).
\end{equation*}
Takeing
$$\lambda^{m}=\frac{B}{4A_{6}}.$$
From the fact that $\|\tilde{\phi}_{k}\|_{L^{\infty}}, \|\tilde{\psi}_{k}\|_{L^{\infty}}\lesssim 1$, it is easy to obtain the desired result for
$$0<\lambda\leq \max\{\|\tilde{\phi}_{1}\|_{L^{\infty}},\cdots, \|\tilde{\phi}_{m}\|_{L^{\infty}},\|\tilde{\psi}_{1}\|_{L^{\infty}},\cdots, \|\tilde{\psi}_{m}\|_{L^{\infty}}\}.$$
It remains to dealt with the case $\lambda>\max\{\|\tilde{\phi}_{1}\|_{L^{\infty}},\cdots, \|\tilde{\phi}_{m}\|_{L^{\infty}},\|\tilde{\psi}_{1}\|_{L^{\infty}},\cdots, \|\tilde{\psi}_{m}\|_{L^{\infty}}\}$. Since
  \begin{equation*}
    \begin{aligned}
      &\nu_{\vec{\omega}}(\{x\in Q_{l}: |[b,T]_1(\tilde{\phi}_1,\cdots,\tilde{\phi}_m)(x)|>\lambda^m\})\\
      &\lesssim
   \prod_{k=1}^{m} \left(\int_{Q_{k,l}}\Phi(\frac{\tilde{\phi}_{k}}{\lambda})\omega_k(y)dy\right)^{1/m}
    \lesssim \frac{\prod_{k=1}^{m}\omega_{k}^{1/m}(Q_{l})}{|Q_{l}|^{m-1}\lambda}.
    \end{aligned}
  \end{equation*}
  Similarly,
  \begin{equation*}
      \nu_{\vec{\omega}}(\{x\in Q_{l}: |[b,T]_1(\tilde{\psi}_1,\cdots,\tilde{\psi}_m)(x)|>\lambda^m\})
      \lesssim \frac{\prod_{k=1}^{m}\omega_{k}^{1/m}(Q_{l})}{|Q_{l}|^{m-1}\lambda}.
  \end{equation*}
Combining with the above estimates, we conclude that
  \begin{equation*}
  \begin{aligned}
  \nu_{\vec{\omega}}(Q_{l})=& \nu_{\vec{\omega}}(\{x\in Q_{l}: A_{6}^{-1}B>2\lambda^{m}\})\\
   \leqslant& \nu_{\vec{\omega}}(\{x\in Q_{l}: |[b,T]_1(\tilde{\phi}_1,\cdots,\tilde{\phi}_m)(x)|+|[b,T]_1(\tilde{\psi}_1,\cdots,\tilde{\psi}_m)(x)|>2\lambda^{m}\})
    \\
  \leqslant&
    \nu_{\vec{\omega}}(\{x\in Q_{l}: |[b,T]_1(\tilde{\phi}_1,\cdots,\tilde{\phi}_m)(x)|>\lambda^{m}\})\\
   & +
    \nu_{\vec{\omega}}(\{x\in Q_{l}: |[b,T]_1(\tilde{\psi}_1,\cdots,\tilde{\psi}_m)(x)|>\lambda^{m}\})\\
   \lesssim &\frac{\prod_{k=1}^{m}\omega_{k}^{1/m}(Q_{l})}{|Q_{l}|^{m-1}\lambda}.
  \end{aligned}
  \end{equation*}
Then $B=4A_{6}\lambda^{m}\lesssim 1$ and the proof of Theorem \ref{GC2} is completed.
\end{proof}

\subsection{The linear commutator of multilinear Calder\'{o}n-Zygmund operators}\label{the linear commutators}

It is very interesting to explore the necessity of bounded linear commutators, since the boundedness of the linear commutators $[\Sigma \vec{b},T]$ can not deduce that each term $[b_{i},T]_{i}$ is a bounded operator.

\medskip

{\bf Theorem D} (cf. \cite{WZTadm})\quad Suppose that $K$ is a homogeneous function of degree $-mn$, and there exists a ball $\mathbb{B}\subset \mathbb{R}^{mn}$ such that $1/K$ can be expended to a Fourier series in $\mathbb{B}$. If $1<p_1,\cdots,p_m,{p}<\infty$, $\frac 1p=\frac 1{p_1}+\cdots+\frac 1{p_m}$,
$$b_{1},\cdots,b_{m}\in BMO\,\Longleftrightarrow\, [\Sigma\vec{b},T]:\, L^{p_1}\times\cdots\times L^{p_m}\to L^p.$$

\medskip
Applying the methods in \cite{WZTadm}, we can obtain the weighted result as Theorem D when $\vec{\omega}\in A_{\vec{P}}$ with $\omega_{k}\in A_{p_{k}}, k=1,\cdots,m$. In order to avoid cumbersome procedure, we omit the details. We try to get the related results with the condition $\vec{\omega}\in A_{\vec{P}}$ only. Unfortunately, we only obtain the case that $b_{1}=\cdots=b_{m}$ and $\frac{1}{m}\leq p\leq \frac{1}{m-1}$. However, the index $p\in [\frac{1}{m},\frac{1}{m-1}]$ is sharp in some sense (see Appendix \ref{A}).

To obtain the desired result, we need the following lemmas.

\begin{lemma}\label{weights-lem2}
Let $1< p_{1},\cdots,p_{m}<\infty$ with $m-1\leq \frac{1}{p}=\frac{1}{p_{1}}+\cdots+\frac{1}{p_{m}}<m$ and $\vec{\omega}=(\omega_{1},\cdots,\omega_{m})\in A_{\vec{P}}$. Then for any cube $Q$ and $k=1,\cdots, m$, $\nu_{\vec{\omega}}^{\frac{1}{1-mp}}\chi_{Q}\in L^{p_{k}}(\omega_{k})$ and
$$\prod_{k=1}^{m}\|\nu_{\vec{\omega}}^{\frac{1}{1-mp}}\chi_{Q}\|_{L^{p_{k}}(\omega_{k})}\lesssim |Q|^{\frac{m}{mp-1}}\nu_{\vec{\omega}}(Q)^{\frac{1}{p(1-mp)}}.$$
\end{lemma}
\begin{proof}
{\bf Case 1. $\frac{1}{p}> m-1$.} It is easy to see that $mp-1>0$ and $p+1-mp>0$. Let
$$\quad s_{k,k}=\frac{(mp-1)p'_{k}}{(p+1-mp)p_{k}}>0 \qquad \text{and} \qquad s_{k,j}=\frac{(mp-1)p'_{j}}{p_{k}p}>0,\quad  j\neq k.$$
Then $\frac{1}{s_{k,1}}+\cdots+\frac{1}{s_{k,m}}=1$, it follows that $s_{k,j}>1$ for any $j=1,\cdots,m$, and
$$\|\nu_{\vec{\omega}}^{\frac{1}{1-mp}}\chi_{Q}\|_{L^{p_{k}}(\omega_{k})}\leq \prod_{j=1}^{m}\bigg(\int_{Q}\omega_{j}^{1-p'_{j}}(x)dx\bigg)^{\frac{1}{p_{i}s_{k,j}}}.$$
Therefore,
\begin{align*}
\prod_{k=1}^{m}\|\nu_{\vec{\omega}}^{\frac{1}{1-mp}}\chi_{Q}\|_{L^{p_{k}}(\omega_{k})}&\leq \prod_{j=1}^{m}\bigg(\int_{Q}\omega_{j}^{1-p'_{j}}(x)dx\bigg)^{\frac{1}{p'_{j}(mp-1)}}\lesssim |Q|^{\frac{m}{mp-1}}\nu_{\vec{\omega}}(Q)^{\frac{1}{p(1-mp)}}.
\end{align*}
{\bf Case 2. $\frac{1}{p}= m-1$.} Let
$$ s_{k,j}=\frac{(mp-1)p'_{j}}{p_{k}p}>0,\quad  j\neq k.$$
Note that
$$\frac{1}{s_{k,1}}+\cdots+\frac{1}{s_{k,k-1}}+\frac{1}{s_{k,k+1}}+\cdots+\frac{1}{s_{k,m}}=1,$$
it follows that $s_{k,j}>1$ for any $j=1,\cdots,m$, and
$$\|\nu_{\vec{\omega}}^{\frac{1}{1-mp}}\chi_{Q}\|_{L^{p_{k}}(\omega_{k})}\leq \prod_{j\neq k}\bigg(\int_{Q}\omega_{j}^{1-p'_{j}}(x)dx\bigg)^{\frac{1}{p_{i}s_{k,j}}}.$$
Therefore,
\begin{align*}
\prod_{k=1}^{m}\|\nu_{\vec{\omega}}^{\frac{1}{1-mp}}\chi_{Q}\|_{L^{p_{k}}(\omega_{k})}&\leq \prod_{j=1}^{m}\bigg(\int_{Q}\omega_{j}^{1-p'_{j}}(x)dx\bigg)^{\frac{1}{p'_{j}(mp-1)}}\lesssim |Q|^{\frac{m}{mp-1}}\nu_{\vec{\omega}}(Q)^{\frac{1}{p(1-mp)}}.
\end{align*}
Thus we complete the proof of Lemma \ref{weights-lem2}.
\end{proof}

For the endpoint case $p_{1}=\cdots=p_{m}=1$, we have $p=1/m$ and
\begin{lemma}\label{weights-lem3}
Let $\vec{\omega}=(\omega_{1},\cdots,\omega_{m})\in A_{(1,\cdots,1)}$. Then for any cube $Q$, we have $\nu_{\vec{\omega}}^{-m}\chi_{Q}\in L^{1}(\omega_{k}), k=1,\cdots, m,$ and
$$\prod_{k=1}^{m}\|\nu_{\vec{\omega}}^{-m}\chi_{Q}\|_{L^{1}(\omega_{k})}\lesssim |Q|^{m}\Big(\frac{|Q|}{\nu_{\vec{\omega}}(Q)}\Big)^{m(m-1)}.$$
\end{lemma}
\begin{proof}
It is easy to see that
\begin{align*}
\|\nu_{\vec{\omega}}^{-m}\chi_{Q}\|_{L^{1}(\omega_{k})}&\leq |Q|\prod_{j\neq k}\|\omega_{j}^{-1}\chi_{Q}\|_{L^{\infty}}.
\end{align*}
Therefore,
\begin{align*}
\prod_{k=1}^{m}\|\nu_{\vec{\omega}}^{-m}\chi_{Q}\|_{L^{1}(\omega_{k})}&\leq |Q|^{m}\prod_{j=1}^{m}\|\omega_{j}^{-1}\chi_{Q}\|_{L^{\infty}}^{m-1}\lesssim |Q|^{m}\Big(\frac{|Q|}{\nu_{\vec{\omega}}(Q)}\Big)^{m(m-1)}.
\end{align*}
Thus we complete the proof of Lemma \ref{weights-lem3}.
\end{proof}

\begin{theorem}\label{LC}
Let $m\in \mathbb{N}$, $\vec{b}=(b,\cdots,b)$ and $T$ be a $m$-linear Calder\'{o}n-Zygmund operator with kernel $K$ and there exists a ball $\mathbb{B}\subset \mathbb{R}^{mn}$ such that $1/K$ can be expended to a Fourier series in $\mathbb{B}$. Then the following four statements are equivalent:
\begin{enumerate}
\item [\rm(1)] $b\in BMO$.
\item [\rm(2)] For $1<p_{1},\cdots,p_{m}<\infty$ with $m-1\leq 1/p=1/p_{1}+\cdots+1/p_{m}<m$ and $\vec{\omega}\in A_{\vec{P}}$,
$$[\Sigma\vec{b},T]: L^{p_{1}}(\omega_{1})\times\cdots\times L^{p_{m}}(\omega_{m})\rightarrow L^{p}(\nu_{\vec{\omega}}).$$
\item [\rm(3)] For $1<p_{1},\cdots,p_{m}<\infty$ with $m-1\leq 1/p=1/p_{1}+\cdots+1/p_{m}<m$ and $\vec{\omega}\in A_{\vec{P}}$,
$$[\Sigma\vec{b},T]: L^{p_{1}}(\omega_{1})\times\cdots\times L^{p_{m}}(\omega_{m})\rightarrow L^{p,\infty}(\nu_{\vec{\omega}}).$$
\item [\rm(4)] For $\vec{\omega}\in A_{1,\cdots,1}$ and $\Phi(t)=t(1+\log^{+}t)$,
$$\nu_{\vec{\omega}}\big(\{x\in \mathbb{R}^{n}:|[\Sigma\vec{b},T](f_{1},\cdots,f_{m})|>\lambda^{m}\}\big)\lesssim \prod_{i=1}^{m}\Big(\int_{\mathbb{R}^n}\Phi(\frac{|f(y)|}{\lambda})\omega_{i}(y)dy\Big)^{1/m}$$
for any $\lambda>0$.
\end{enumerate}
\end{theorem}

\begin{proof}
Let
\begin{equation*}
    \left\{
   \begin{array}{ll}\vspace{1ex}
\mu:=\nu_{\vec{\omega}}^{-m}, & p_{1}=\cdots=p_{m}=1,\\
\mu:=\nu_{\vec{\omega}}^{\frac{1}{1-mp}}, & 1<p_{1},\cdots,p_{m}<\infty,
   \end{array}
 \right.
\end{equation*}
and let $s(x)=\overline{\mathrm{sgn}(\int_{Q'}(b(x)-b(y))\mu(y)dy)}$. Then for any $j\in \{1,\cdots,m\}$,
\begin{eqnarray*}
\begin{aligned}
&|b(x)-b_{\mu,Q'}|
=\frac{s(x)}{\mu(Q')^{m}}\int_{(Q')^{m}}\big(b(x)-b(y_{j})\big)\prod_{k=1}^{m}\mu(y_{k}) d\vec{y}.
\end{aligned}
\end{eqnarray*}
Define the functions
\begin{equation*}
    \left\{
   \begin{array}{ll}\vspace{1ex}
   g^{k}_{l}(y_{k})=e^{-\mathrm{i}\frac{\delta}{r}v^{k}_{l}\cdot y_{k}}\mu(y_{k})\chi_{Q'}(y_{k})\frac{|Q'|}{\mu(Q')}, k=1,\cdots,m, \\
   h_{l}(x)=e^{\mathrm{i}\frac{\delta}{r}v_{l}\cdot (x,\cdots,x)}\chi_{Q}(x)s(x),
\end{array}
 \right.
\end{equation*}
which shows that
\begin{align*}
m|b(x)-b_{\mu,Q'}|&= s(x)\frac{r^{mn}\delta^{-mn}}{\mu(Q')^{m}}
\int_{(Q')^{m}}\sum_{j=1}^{m}(b(x)-b(y_{j}))K(x-y_{1},\cdots,x-y_{m})\\
&\qquad\times\sum_{l}a_{l}e^{\mathrm{i}\frac{\delta}{r}v_{l}\cdot(x-y_{1},\cdots,x-y_{m})}\prod_{k=1}^{m}\mu(y_{k})d\vec{y}\\
&=\delta^{-mn}\sum_{l}a_{l}[\Sigma\vec{b},T](g^{1}_{l},\cdots,g^{m}_{l})(x)h_{l}(x).
\end{align*}

{\bf Case 1. $1<p_{1},\cdots, p_{m}<\infty$.} Note that $\nu_{\vec{\omega}}\in A_{mp}$ and let $s_{k}=(m-1/p)p'_{k}$ with $1/s_{1}+\cdots+1/s_{m}=1$. Therefore,
\begin{align*}
|Q|&=\int_{Q}\nu_{\vec{\omega}}(x)^{1/mp}\cdot \nu_{\vec{\omega}}(x)^{-1/mp}dx\\
&\leq \nu_{\vec{\omega}}(Q)^{\frac{1}{mp}}\Big(\int_{Q}\nu_{\vec{\omega}}(x)^{\frac{1}{1-mp}}dx\Big)^{\frac{mp-1}{mp}}\\
&\leq \nu_{\vec{\omega}}(Q)^{\frac{1}{mp}}\prod_{k=1}^{m}\Big(\int_{Q}\omega_{k}(x)^{1-p'_{k}}dx\Big)^{\frac{1}{mp'_{k}}}.
\end{align*}
The doubling property of $\mu_{1},\cdots,\mu_{m}$ to deduce that
\begin{align*}
&\inf_{c}\frac{\lambda}{\nu_{\vec{\omega}}(Q)^{1/p}}\nu_{\vec{\omega}}\big(x\in Q:|b(x)-c|>\lambda\big)^{1/p}\\
&\lesssim \frac{\lambda}{\nu_{\vec{\omega}}(Q)^{1/p}}\nu_{\vec{\omega}}\big(x\in Q:m|b(x)-b_{\mu,Q'}|>m\lambda\big)^{1/p}\\
&\lesssim \frac{\lambda}{\nu_{\vec{\omega}}(Q)^{1/p}}\nu_{\vec{\omega}}\big(x\in Q:\sum_{l}|a_{l}|\big|[\Sigma\vec{b},T](g_{1}^{1},\cdots,g_{l}^{m})(x)\big|>m\lambda\big)^{1/p}.
\end{align*}
By Lemma \ref{weights-lem2}, we get
\begin{align*}
\inf_{c}\frac{\lambda}{\nu_{\vec{\omega}}(Q)^{1/p}}\nu_{\vec{\omega}}\big(x\in Q:|b(x)-c|>\lambda\big)^{1/p}\lesssim \sum_{l}|a_{l}|.
\end{align*}

{\bf Case 2. $p_{1}=\cdots= p_{m}=1$.} It follows from $\nu_{\vec{\omega}}\in A_{1}$ that
$$\|\mu\chi_{Q'}\|_{L^{\infty}}=\|\nu_{\vec{\omega}}^{-m}\chi_{Q'}\|_{L^{\infty}}\lesssim \Big(\frac{|Q'|}{\nu_{\vec{\omega}}(Q')}\Big)^{m}.$$
From the fact that
\begin{eqnarray*}
\begin{aligned}
|Q'|&=\int_{Q'}\nu_{\vec{\omega}}(x)^{\frac{1}{2}}\cdot\nu_{\vec{\omega}}(x)^{-\frac{1}{2}}dx\\
&\lesssim |Q'|^{\frac{m-1}{2m}}\Big(\int_{Q'}\nu_{\vec{\omega}}(x)dx\Big)^{\frac{1}{2}}\Big(\int_{Q'}\nu_{\vec{\omega}}(x)^{-m}dx\Big)^{\frac{1}{2m}},
\end{aligned}
\end{eqnarray*}
we have
 $$\frac{|Q'|}{\mu(Q')}=\frac{|Q'|}{\int_{Q'}\nu_{\vec{\omega}}(x)^{-m}dx}\lesssim \big(\frac{\nu_{\vec{\omega}}(Q')}{|Q'|}\big)^{m}.$$
Combining the above computations we obtain $\|g^{1}_{l}\|_{L^{\infty}}=\cdots=\|g^{m}_{l}\|_{L^{\infty}}\lesssim 1$.

If $0<\lambda<\|g^{1}_{l}\|_{L^{\infty}}\big(\sum_{l}|a_{l}|\big)^{1/m}$, it is easy to see that
\begin{align*}
\inf_{c}\frac{\lambda}{\nu_{\vec{\omega}}(Q)}\nu_{\vec{\omega}}\big(x\in Q:|b(x)-c|>\lambda^{m}\big)\lesssim \lambda\lesssim \big(\sum_{l}|a_{l}|\big)^{1/m}.
\end{align*}

If $\|g^{1}_{l}\|_{L^{\infty}}\big(\sum_{l}|a_{l}|\big)^{1/m} \leq \lambda<\infty$, then $$\Phi\Big(\frac{\big(\sum_{l}|a_{l}|\big)^{1/m}|g^{k}_{l}(y_{k})|}{\lambda}\Big)=\frac{|g^{k}_{l}(y_{k})|\big(\sum_{l}|a_{l}|\big)^{1/m}}{\lambda}.$$
We can now continue the above estimate:
\begin{align*}
\inf_{c}\nu_{\vec{\omega}}\big(x\in Q:|b(x)-c|>\lambda^{m}\big)
&\leq \nu_{\vec{\omega}}\big(x\in Q:|b(x)-b_{\mu,Q'}|>\lambda^{m}\big)\\
&\lesssim \nu_{\vec{\omega}}\big(x\in Q:\sum_{l}|a_{l}|\big|[\Sigma\vec{b},T](g^{1}_{l},\cdots,g^{m}_{l})(x)\big|>m\lambda^{m}\big)\\
&\lesssim \prod_{k=1}^{m}\Big(\int_{Q'}\Phi\Big(\frac{|g^{k}_{l}(y_{k})|\big(\sum_{l}|a_{l}|\big)^{1/m}}{\lambda}\Big)\omega_{k}(y_{k})dy_{k}\Big)^{1/m}.
\end{align*}
By Lemma \ref{weights-lem3}, we have
\begin{align*}
\prod_{k=1}^{m}\Big(\int_{Q'}|g^{k}_{l}(y_{k})|\omega_{k}(y_{k})dy_{k}\Big)^{1/m}&\lesssim |Q'|\Big(\frac{|Q'|}{\nu_{\vec{\omega}}(Q')}\Big)^{m-1}\cdot \Big(\frac{\nu_{\vec{\omega}}(Q')}{|Q'|}\Big)^{m}\lesssim \nu_{\vec{\omega}}(Q').
\end{align*}
Combining with the above estimates, we conclude that
\begin{align*}
\inf_{c}\frac{\lambda}{\nu_{\vec{\omega}}(Q)}\nu_{\vec{\omega}}\big(x\in Q:|b(x)-c|>\lambda^{m}\big)\lesssim \big(\sum_{l}|a_{l}|\big)^{1/m}.
\end{align*}
From Lemma \ref{BMO-lem2}, we have that $b\in BMO$ and the proof is complete.
\end{proof}

\subsection{The iterated commutator of multilinear Calder\'{o}n-Zygmund operators}\label{the iterated commutators}

Much of the analysis of linear commutators has been extended to other context such as weighted spaces, multiparameter and multilinear settings. Iterated commutators have been considered too.
The boundedness of the commutator with the symbol functions $BMO$ has been extensively studied already. In this subsection, we study the necessity theory for iterated commutators.

First, we find illustrative to present some functions, which shed light on the boundedness of iterated commutators of multilinear Calder\'{o}n-Zygmund operators. The boundedness of the commutator with the symbol functions $BMO$ and $\mathcal{C}_{\alpha,q} (\alpha>0)$ has been extensively studied already. Unlike the case $\alpha\geq 0$ (the spaces $\mathcal{C}_{\alpha,q}$ are independent of the scale $0<q<\infty$), more difficulties are caused for the case $\alpha<0$. Then, some changes are needed to deal with the boundedness of the commutator with the symbol functions $\mathcal{C}_{\alpha,q}$ with $\alpha<0$.

\begin{lemma}\label{lem IC}
Let $1<q_{1},\cdots,q_{m}<\infty$, $-\frac{n}{q_{m}}<\alpha_{m}<0<\alpha_{1},\cdots,\alpha_{m-1}<1$, $\alpha_{1}+\cdots+\alpha_{m}=0$ and $b_{i}\in \mathcal{C}_{\alpha_{i},q_{i}}$ with $i=1,\cdots,m$. If $\gamma+\alpha_{m}>0$, there exists $\delta>0$ such that
\begin{align*}
M^{\sharp}_{\delta}\big([\Pi \vec{b},T](\vec{f})\big)(x)\lesssim& \mathcal{M}_{s}(\vec{f})(x),
\end{align*}
for any $s>q'_{m}$ and bounded compact supported functions $f_{1},\cdots,f_{m}$.
\end{lemma}

\begin{proof}
For simplicity in notation and proof, we only give the arguments for $m=2$ and $b_{1}\in \mathcal{C}_{\alpha,q_{1}}, b_{2}\in \mathcal{C}_{-\alpha,q_{2}}$ with $\alpha>0$, since the other cases can be treated similarly.

Let $B:=B(x_{0},r)$ be a ball with $x\in B$ and $\Omega_{x_{0},r}:=\big\{(y_{1},y_{2}):|x_{0}-y_{1}|+|x_{0}-y_{2}|\leq r\big\}$. Then, $\Omega_{x_{0},4r}\subset B(x_{0},4r)\times B(x_{0},4r)$ for any $z\in B$ we have $$\big|[\Pi \vec{b},T](f_{1},f_{2})(z)-c_{B}\big|=\mathrm{I}^{B}_{1}(z)+\mathrm{II}^{B}_{2}(z)$$ with
\begin{align*}
&\mathrm{I}(z)=\int_{\Omega_{x_{0},4r}}(b_{1}(z)-b_{1}(y_{1}))(b_{2}(z)-b_{1}(y_{2}))K(z,y_{1},y_{2})f_{1}(y_{1}) f_{2}(y_{2})dy_{1}dy_{2},\\
&\mathrm{II}(z)=\int_{\Omega^{c}_{x_{0},4r}}(b_{1}(z)-b_{1}(y_{1}))(b_{2}(z)-b_{1}(y_{2}))K(z,y_{1},y_{2})f_{1}(y_{1}) f_{2}(y_{2})dy_{1}dy_{2}-c_{B},\\
&c_{B}=\frac{1}{|B|}\int_{B}\int_{\Omega^{c}_{x_{0},4r}}(b_{1}(z)-b_{1}(y_{1}))(b_{2}(z)-b_{1}(y_{2}))K(z,y_{1},y_{2})f_{1}(y_{1}) f_{2}(y_{2})dy_{1}dy_{2}dz.
\end{align*}
Therefore,
\begin{align*}
&\bigg(\frac{1}{|B|}\int_{B}\Big|\big|[\Pi \vec{b},T](\vec{f})(z)\big|^{\delta}-|c_{B}|^{\delta}\Big|dz\bigg)^{1/\delta}\lesssim \bigg(\frac{1}{|B|}\int_{B}\Big|[\Pi \vec{b},T](\vec{f})(z)-c_{B}\Big|^{\delta}\bigg)^{1/\delta}\lesssim \mathrm{I+II},
\end{align*}
where $\mathrm{I}=\big(\frac{1}{|B|}\int_{B}\big|\mathrm{I}(z)\big|^{\delta}dz\big)^{1/\delta}$ and $\mathrm{II}=\big(\frac{1}{|B|}\int_{B}\big|\mathrm{II}(z)\big|^{\delta}dz\big)^{1/\delta}$.

Let us consider first the term $\mathrm{I}$. By $|b_{1}(z)-b_{1}(y_{1})|\lesssim \|b_{1}\|_{\mathcal{C}_{\alpha,q_{1}}}|z-y_{1}|^{\alpha}$ for $y_{1}\in 2B,z\in B$, one immediately obtains
\begin{align*}
\mathrm{I}\lesssim&\bigg(\frac{1}{|B|}\int_{B}\Big(\big|b_{2}(z)-b_{2,B}\big|I_{\alpha}(|f^{0}_{1}|,|f^{0}_{2}|)(z)\Big)^{\delta}dz\bigg)^{1/\delta}\\
&+\bigg(\frac{1}{|B|}\int_{B}I_{\alpha}(|f_{1}^{0}|,|b_{2}-b_{2,B}||f_{2}^{0}|)(z)\Big)^{\delta}dz\bigg)^{1/\delta}\\
&=:\mathrm{I}_{1}+\mathrm{I}_{2},
\end{align*}
where $f_{i}^{0}=f_{i}\chi_{2B}, i=1,2$ and
$$I_{\alpha}(f_{1},f_{2})(x)=\int_{\mathbb{R}^{2n}}\frac{f_{1}(y_{1})f_{2}(y_{2})}{\big(|x-y_{1}|+|x-y_{2}|\big)^{2n-\alpha}}dy_{1}dy_{2}.$$
H\"{o}lder inequality, Kolmogorov inequality and the weak boundedness of $I_{\alpha}$ yield that
\begin{align*}
\mathrm{I}_{1}&\lesssim \bigg(\frac{1}{|B|}\int_{B}\big|b_{2}(z)-b_{2,B}\big|^{q_{2}}dz\bigg)^{\frac{1}{q_{2}}} \bigg(\frac{1}{|B|}\int_{B}\big|I_{\alpha}(|f_{1}^{0}|,|f^{0}_{2}|)(x)\big|^{t}dz\bigg)^{\frac{1}{t}}|B|^{-\alpha/n}\\
&\lesssim |B|^{-2}\|b_{2}\|_{\mathcal{C}_{-\alpha,q_{2}}}\|I_{\alpha}(|f_{1}^{0}|,|f^{0}_{2}|)\|_{L^{\frac{n}{2n-\alpha}},\infty}\\
&\lesssim \mathcal{M}(f_{1},f_{2})(x),
\end{align*}
where $\delta<t=\frac{\delta^{2}q_{2}}{\delta q_{2}-1}<\frac{n}{2n-\alpha}$. Similarly,
\begin{align*}
\mathrm{I}_{2}\lesssim |B|^{-2+\alpha/n}\|I_{\alpha}(|f_{1}^{0}|,|b_{2}-b_{2,B}||f^{0}_{2}|)\|_{L^{\frac{n}{2n-\alpha}},\infty}\lesssim \mathcal{M}_{s}(f_{1},f_{2})(x),
\end{align*}
where $q'_{2}\leq s<\infty$. For $\mathrm{II}$, due to $|z-z'|\leq \frac{1}{2}\big(|z-y_{1}|+|z-y_{2}|\big)$, then
$$\big|K(z,y_{1},y_{2})-K(z',y_{1},y_{2})\big|\lesssim \frac{|z-z'|^{\gamma}}{\big(|z-y_{1}|+|z-y_{2}|\big)^{2n+\gamma}}.$$
Therefore,
\begin{align*}
\mathrm{II}(z)\lesssim& |B|^{\gamma/n}\int_{\Omega^{c}_{x_{0},4r}}\frac{|b_{1}(z)-b_{1}(y_{1})||b_{2}(z)-b_{2}(y_{2})||f_{1}(y_{1})||f_{2}(y_{2})|}{\big(|z-y_{1}|+|z-y_{2}|\big)^{2n+\gamma}}dy_{1}dy_{2}\\
\lesssim& |B|^{\gamma/n}\sum_{k=2}^{\infty}\int_{\Omega_{2^{k+1}r}\backslash\Omega_{2^{k}r}}\frac{|b_{1}(z)-b_{1}(y_{1})||b_{2}(z)-b_{2}(y_{2})||f_{1}(y_{1})||f_{2}(y_{2})|}{\big(|z-y_{1}|+|z-y_{2}|\big)^{2n+\gamma}}dy_{1}dy_{2}\\
\lesssim& \sum_{k=2}^{\infty}\frac{|B|^{\gamma/n}}{|2^{k}B|^{2+\gamma/n}}\int_{\Omega_{2^{k+1}r}\backslash\Omega_{2^{k}r}}|b_{1}(z)-b_{1}(y_{1})||b_{2}(z)-b_{2}(y_{2})||f_{1}(y_{1})||f_{2}(y_{2})|dy_{1}dy_{2}\\
\lesssim& \sum_{k=2}^{\infty}\frac{|B|^{\gamma/n}|b_{2}(z)-b_{2,B}|}{|2^{k}B|^{2+\gamma/n-\alpha/n}}\int_{2^{k+1}B}\int_{2^{k+1}B}|f_{1}(y_{1})||f_{2}(y_{2})|dy_{1}dy_{2}\\
&+ \sum_{k=2}^{\infty}\frac{|B|^{\gamma/n}}{|2^{k}B|^{2+\gamma/n-\alpha/n}}\int_{2^{k+1}B}\int_{2^{k+1}B}|b_{2,B}-b_{2}(y_{2})||f_{1}(y_{1})||f_{2}(y_{2})|dy_{1}dy_{2}.
\end{align*}
From $\gamma>\alpha$ and $s>q'_{2}$, it is easy to see that
$$\mathrm{II}\lesssim \mathcal{M}_{s}(f_{1},f_{2})(x).$$
Collecting our estimates, we have shown that
\begin{align*}
M^{\sharp}_{\delta}\big([\Pi\vec{b},T](\vec{f})\big)(x)\lesssim& \mathcal{M}_{s}(\vec{f})(x)
\end{align*}
for any $s>q'_{2}$ and bounded compact supported functions $f_{1},f_{2}$.
\end{proof}

\begin{theorem}\label{IC-1}
Let $1<q_{1},\cdots,q_{m}<\infty$, $-\frac{n}{q_{m}}<\alpha_{m}<0<\alpha_{1},\cdots,\alpha_{m-1}<1$, $\alpha_{1}+\cdots+\alpha_{m}=0$ and $b_{i}\in \mathcal{C}_{\alpha_{i},q_{i}}$ with $i=1,\cdots,m$. If $\gamma+\alpha_{m}>0$,  $\omega_{k}\in A_{p_{k}}$ and $q_{m}$ is large enough, then $[\Pi\vec{b},T]$ is bounded from $L^{p_{1}}(\omega_{1})\times\cdots\times L^{p_{m}}(\omega_{m})$ to $L^{p}(\nu_{\vec{\omega}})$.
\end{theorem}

\begin{proof}
As stated in \cite{PPTT2014}, the proof of these types of estimates is standard. See \cite[Theorems 1.1 and 1.6]{PT2003} for the linear cases, and, \cite[Theorem 3.19]{LOPTT2009}, \cite[Theorem 3.2]{PPTT2014} for multilinear cases, for example.

For $\omega_{k}\in A_{p_{k}}$, there exists a constant $\epsilon_{k}>0$ such that $\omega_{k}\in A_{p_{k}-\epsilon_{k}}$. There exists $\epsilon>0$ small enough such that $\omega_{k}\in A_{\tilde{p}_{k}}$, where $\tilde{p}_{k}:=\frac{p_{k}}{1+\epsilon}$. By Lemma \ref{lem IC} with $s=1+\epsilon$, from a standard argument that we can obtain
\begin{align*}
&\|[\Pi\vec{b},T](\vec{f})\|_{L^{p}(\nu_{\vec{\omega}})}\lesssim \|M_{\delta}\big([\Pi\vec{b},T](\vec{f})\big)\|_{L^{p}(\nu_{\vec{\omega}})}
\lesssim\|M^{\sharp}_{\delta}\big([\Pi\vec{b},T](\vec{f})\big)\|_{L^{p}(\nu_{\vec{\omega}})}\\
&\lesssim\big\|\mathcal{M}_{s}(\vec{f})\big\|_{L^{p}(\nu_{\vec{\omega}})}\lesssim \prod_{k=1}^{m}\|M_{s}(f_{k})\|_{L^{p_{k}}(\omega_{k})}
\lesssim \prod_{k=1}^{m}\big\||f_{k}|^{s}\big\|^{1/s}_{L^{p_{k}/s}(\omega_{k})}= \prod_{k=1}^{m}\|f_{k}\|_{L^{p_{k}}(\omega_{k})}.
\end{align*}
Thus, the proof of Theorem \ref{IC-1} is completed.
\end{proof}

Theorem \ref{IC-1} implies that the symbol function belongs to $BMO$ space is not necessary condition for the boundedness of iterated commutator acting on product of (weighted) Lebesgue spaces. However, under certain condition, we have the following conclusion.
\medskip

{\bf Theorem E} (cf. \cite{WZTamp})\quad Suppose that $K$ is a homogeneous function of degree $-mn$, and there exists a ball $\mathbb{B}\subset \mathbb{R}^{mn}$ such that $1/K$ can be expended to a Fourier series in $\mathbb{B}$. If $1<p_1,\cdots,p_m<\infty$, $\frac 1p=\frac 1{p_1}+\cdots+\frac 1{p_m}$ and $b_{1}=b_{2}=\cdots=b_{m}$, then
$$b\in BMO\,\Longleftrightarrow\, [\Pi\vec{b},T]:\, L^{p_1}\times\cdots\times L^{p_m}\to L^p.$$

\medskip

Some of the techniques employed in \cite{WZTamp} do not apply to the weighted result, because the weights are not locally integrable. Therefore, we can only obtain

\begin{theorem}\label{IC-2}
Let $m\in \mathbb{N}$, $\vec{b}=(b,\cdots,b)$ and $T$ be a $m$-linear Calder\'{o}n-Zygmund operator with kernel $K$ and there exists a cube $Q\subset \mathbb{R}^{mn}$ such that $1/K$ can be expended to a Fourier series in $Q$. Then the following three statements are equivalent:
\begin{enumerate}
\item [\rm(1)] $b\in BMO$.
\item [\rm(2)] For $1<p_{1},\cdots,p_{m}<\infty$ with $m-1\leq 1/p=1/p_{1}+\cdots+1/p_{m}<m$ and $\vec{\omega}\in A_{\vec{P}}$,
$$[\Pi\vec{b},T]: L^{p_{1}}(\omega_{1})\times\cdots\times L^{p_{m}}(\omega_{m})\rightarrow L^{p}(\nu_{\vec{\omega}}),\qquad i=1,\cdots,m.$$
\item [\rm(3)] For $1<p_{1},\cdots,p_{m}<\infty$ with $m-1\leq 1/p=1/p_{1}+\cdots+1/p_{m}<m$ and $\vec{\omega}\in A_{\vec{P}}$,
$$[\Pi\vec{b},T]: L^{p_{1}}(\omega_{1})\times\cdots\times L^{p_{m}}(\omega_{m})\rightarrow L^{p,\infty}(\nu_{\vec{\omega}}),\qquad i=1,\cdots,m.$$
\item [\rm(4)] For $\vec{\omega}\in A_{1,\cdots,1}, \Phi(t)=t(1+\log^{+}t)$,
$$\nu_{\vec{\omega}}\big(\{x\in \mathbb{R}^{n}:|[\Pi\vec{b},T](f_{1},\cdots,f_{m})|>\lambda^{m}\}\big)\lesssim \prod_{k=1}^{m}\Big(\int_{\mathbb{R}^n}\Phi^{(m)}(\frac{|f_{k}(y)|}{\lambda})\omega_{k}(y)dy\Big)^{1/m}$$
for any $\lambda>0$.
\end{enumerate}
\end{theorem}

\begin{proof}
We can now estimate as follows. Let
\begin{equation*}
    \left\{
   \begin{array}{ll}\vspace{1ex}
\mu:=\nu_{\vec{\omega}}^{-m}, & p_{1}=\cdots=p_{m}=1,\\
\mu:=\nu_{\vec{\omega}}^{\frac{1}{1-mp}}, & 1<p_{1},\cdots,p_{m}<\infty,
   \end{array}
 \right.
\end{equation*}
and let $s(x)=\overline{\mathrm{sgn}(\int_{Q'}(b(x)-b(y))\mu(y)dy)}$. Then $\mu\in A_{\infty}$ and for any $j=1,\cdots,m$,
\begin{eqnarray*}
\begin{aligned}
&|b(x)-b_{\mu,Q'}|
=\frac{s(x)}{\mu(Q')^{m}}\int_{(Q')^{m}}\big(b(x)-b(y_{j})\big)\prod_{k=1}^{m}\mu(y_{k}) d\vec{y}.
\end{aligned}
\end{eqnarray*}
Define the functions
$$h_{l}(x)=e^{\mathrm{i}\frac{\delta}{r}v_{l}\cdot (x,\cdots,x)}\chi_{Q}(x)s(x)^{m},\qquad g^{k}_{l}(y_{l})=e^{-\mathrm{i}\frac{\delta}{r}v^{k}_{l}\cdot y_{k}}\mu(y_{k})\chi_{Q'}(y_{k})\frac{|Q'|}{\mu(Q')}, \quad k=1,\cdots,m.$$
This shows that
\begin{align*}
|b(x)-b_{\mu,Q'}|^{m}&= s(x)^{m}\frac{r^{mn}\delta^{-mn}}{\mu(Q')^{m}}
\int_{(Q')^{m}}\prod_{k=1}^{m}(b(x)-b(y_{j}))K(x-y_{1},\cdots,x-y_{m})\\
&\qquad\times\sum_{l}a_{l}e^{\mathrm{i}\frac{\delta}{r}v_{l}\cdot(x-y_{1},\cdots,x-y_{m})}\prod_{k=1}^{m}\mu(y_{k})d\vec{y}\\
&=\delta^{-mn}\sum_{l}a_{l}[\Pi\vec{b},T](g^{1}_{l},\cdots,g^{m}_{l})(x)h_{l}(x).
\end{align*}

If $0<\lambda<\|g^{1}_{l}\|_{L^{\infty}}\big(\sum_{l}|a_{l}|\big)^{1/m}$, it is easy to see that
\begin{align*}
\inf_{c}\frac{\lambda}{\nu_{\vec{\omega}}(Q)}\nu_{\vec{\omega}}\big(x\in Q:|b(x)-c|>\lambda^{m}\big)\lesssim \lambda\lesssim \big(\sum_{l}|a_{l}|\big)^{1/m}.
\end{align*}

If $\|g^{1}_{l}\|_{L^{\infty}}\big(\sum_{l}|a_{l}|\big)^{1/m} \leq \lambda<\infty$, then \begin{align*}
\Phi^{(m)}\Big(\frac{\big(\sum_{l}|a_{l}|\big)^{1/m}|g^{k}_{l}(y_{k})|}{\lambda}\Big)=\cdots=\Phi\Big(\frac{\big(\sum_{l}|a_{l}|\big)^{1/m}|g^{k}_{l}(y_{k})|}{\lambda}\Big)
=\frac{|g^{k}_{l}(y_{k})|\big(\sum_{l}|a_{l}|\big)^{1/m}}{\lambda}.
\end{align*}
The same estimates as in the calculation process of Theorem \ref{LC}, we complete the proof of Theorem \ref{IC-2}.
\end{proof}

\appendix

\section{Appendix}\label{A}

\begin{proposition}\label{Kernel}
Let $K$ be a function defined off the diagonal $x=y_1=\cdots =y_m$ in $(\mathbb{R}^n)^{m+1}$, satisfying
\begin{equation*}
  |K(x,y_1,\cdots,y_i,\cdots,y_m)-K(x,y_1,\cdots,y_i',\cdots,y_m)|
  \lesssim \frac{|y_i-y_i'|^{\gamma}}{(\sum_{k=1}^m|x-y_k|)^{mn+\gamma}},
\end{equation*}
whenever $|y_i-y_i'|\leqslant \frac{1}{2}\max_{1\leqslant k\leqslant m}|x-y_k|$.
Then, for any open cone $\widetilde{\Gamma}\subset (\mathbb{R}^n)^m$, there exists an open cone $\Gamma\subset \widetilde{\Gamma}$
such that for every $\{h_l=(h_l^1,\cdots,h_l^m)\}_l\subset \Gamma$ satisfying $|h_l|\rightarrow \infty$ as $l\rightarrow \infty$ and for any $\mu_{1},\cdots,\mu_{m}\in A_{\infty}$, we have
\begin{equation*}
    \begin{aligned}
 |h_l|^{mn}\Big\|\int_{\prod\limits_{{k=1}}^m(Q-\sqrt[n]{Q}h_l^k)}
     & |K(\cdot,y_1,\cdots,y_m)\\
      &-K(\cdot,y_1,\cdots,y_{i-1},a_{Q}-\sqrt[n]{|Q|}h_l^i,y_{i+1},\cdots,y_m)|\prod_{k=1}^{m}\mu_{k}(y_{k})d\vec{y}\Big\|_{L^{\infty}(Q)}
  \rightarrow 0,
    \end{aligned}
  \end{equation*}
  as $l \rightarrow \infty$ uniformly for all cubes $Q$, where $a_Q$ denotes the center of $Q.$
\end{proposition}

\begin{proof}
For any cubes $Q$ and $Q_{k}, k=1,\cdots,m$ such that $|Q|=|Q_k|$ and $|a_Q-a_{Q_k}|/\sqrt[n]{|Q|}$ sufficient large. If $x\in Q$, $y_k\in Q_k$, we have
\begin{equation*}
  |y_k-a_{Q_k}|\leqslant \frac{1}{2}|x-y_k|\leqslant \frac{1}{2}\max_{1\leqslant k\leqslant m}|x-y_k|,
\end{equation*}
which implies that
\begin{equation*}
  \begin{aligned}
    &|K(x,y_1,\cdots,y_m)-K(x,y_1,\cdots,y_{i-1},a_{Q_i},y_{i+1},\cdots,y_m)|\\
   & \lesssim
    \frac{|y_i-a_{Q_i}|^{\gamma}}{(\sum_{k=1}^m|x-y_k|)^{mn+\gamma}}
        \lesssim
    \frac{|Q_i|^{\gamma/n}}{|a_Q-a_{Q_i}|^{mn+\gamma}}.
  \end{aligned}
\end{equation*}
Thus, for any $x\in Q$,
\begin{equation*}
  \begin{aligned}
    &\frac{|a_{Q}-a_{Q_i}|^{mn}}{\prod_{k=1}^{m}\mu_{k}(Q_{k})}\int_{\prod\limits_{{k=1}}^m Q_k}|K(x,y_1,\cdots,y_m)
    -K(x,\cdots,y_{i-1},a_{Q_i},y_{i+1},\cdots,y_m)|\prod_{k=1}^{m}\mu_{k}(y_{k})d\vec{y}\\
    &\qquad\lesssim
    \frac{|a_{Q}-a_{Q_i}|^{mn}}{\prod_{k=1}^{m}\mu_{k}(Q_{k})}\cdot\frac{|Q_i|^{\gamma/n}}{|a_Q-a_{Q_i}|^{mn+\gamma}}\cdot \prod_{k=1}^{m}\mu_{k}(Q_{k})
   =
    \frac{|Q_i|^{\gamma/n}}{|a_Q-a_{Q_i}|^{\gamma}}\rightarrow 0,
  \end{aligned}
\end{equation*}
as $\frac {|a_Q-a_{Q_i}|}{\sqrt[n]{|Q|}}\to \infty$. Thus, we verify that
\begin{equation*}
  \begin{aligned}
    &\left\|\int_{\prod\limits_{{j=1}}^mQ_j}|K(\cdot,y_1,\cdots,y_m)-K(\cdot,y_1,\cdots,y_{i-1},a_{Q_i},y_{i+1},\cdots,y_m)|\prod_{k=1}^{m}\mu_{k}(y_{k})d\vec{y}\right\|_{L^{\infty}(Q)}
    \\
    &\qquad\times    \frac{|a_{Q}-a_{Q_i}|^{mn}}{\prod_{k=1}^{m}\mu_{k}(Q_{k})}\rightarrow 0
  \end{aligned}
\end{equation*}
uniformly for all cubes $Q$, $Q_k$ with $|Q|=|Q_k|,\, k=1,\cdots, m$,
  as $|a_Q-a_{Q_i}|/\sqrt[n]{|Q|}\rightarrow \infty$.

For every open cone $\widetilde{\Gamma}$, choose an open cone $\Gamma\subset \widetilde{\Gamma}$ such that for every
$h=(h^1,\cdots,h^m)\in \Gamma$, we have $|h|\sim |h^i|$, where $h^i\in \mathbb{R}^n$.
Let $Q_i=Q-\sqrt[n]{|Q|}h_l^i$ for every $i=1,\cdots,m$. Using $|a_Q-a_{Q_i}|/\sqrt[mn]{\prod_{k=1}^{m}\mu_{k}(Q_{k})}=|h_l^i|$
and the fact $|h_l|\sim |h_l^i|$, we deduce that
\begin{equation*}
    \begin{aligned}
    |h_l|^{mn}\times  &\Big\|\int_{\prod\limits_{{k=1}}^m(Q-\sqrt[n]{Q}h_l^k)}|K(\cdot,y_1,\cdots,y_m)\\
      &\qquad-K(\cdot,y_1,\cdots,y_{i-1},a_{Q}-\sqrt[n]{|Q|}h_l^i,y_{i+1},\cdots,y_m)|\prod_{k=1}^{m}\mu_{k}(y_{k})d\vec{y}\Big\|_{L^{\infty}(Q)} \rightarrow 0,
    \end{aligned}
\end{equation*}
as $l \rightarrow \infty$ uniformly for all cubes $Q\subset \mathbb{R}^n$. This completes the proof of Proposition \ref{Kernel}.
\end{proof}

\begin{proposition}\label{lem2}
Let $1< p_{1},\cdots,p_{m}<\infty$ with $\frac{1}{p}=\frac{1}{p_{1}}+\cdots+\frac{1}{p_{m}}$.
\begin{enumerate}
  \item If $p\leq \frac{1}{m-1}$, then for any cube $Q$ and $\vec{\omega}=(\omega_{1},\cdots,\omega_{m})\in A_{\vec{P}}$, $\nu_{\vec{\omega}}^{\alpha}\chi_{Q}\in \bigcap_{i=1}^{m}L^{p_{i}}(\omega_{i})$ if and only if $\frac{1}{1-mp}\leq \alpha\leq 1-m.$

 \item If $p> \frac{1}{m-1}$, for any $\alpha\in \mathbb{R}$, there exist $\vec{\omega}=(\omega_{1},\cdots,\omega_{m})\in A_{\vec{P}}$ and a cube $Q$ such that $\nu_{\vec{\omega}}^{\alpha}\chi_{Q}\notin \bigcap_{i=1}^{m}L^{p_{i}}(\omega_{i})$.
\end{enumerate}
\end{proposition}
\begin{proof}
(1) Let $\frac{1}{m}< p\leq \frac{1}{m-1}$ and for any $i=1,\cdots,m,$
$$\quad s_{i,i}=\frac{p'_{i}}{(mp'_{i}-1)(p+1-mp)p_{i}}>0, s_{i,j}=\frac{(mp-1)p'_{j}}{p_{i}p}>0,\quad  j\neq i.$$
In fact, $s_{i,i}=\infty$ if $p=\frac{1}{m-1}$. Then $\frac{1}{s_{i,1}}+\cdots+\frac{1}{s_{i,m}}=1$, it follows that $s_{i,j}>1$ for any $j=1,\cdots,m$, and
$$\|\nu_{\vec{\omega}}^{1-m}\chi_{Q}\|_{L^{p_{i}}(\omega_{i})}\leq \bigg(\int_{Q}\omega_{i}(x)^{\frac{1-p'_{i}}{1-mp'_{i}}}dx\bigg)^{\frac{1}{p_{i}s_{i,i}}}\prod_{j\neq i}\bigg(\int_{Q}\omega_{j}^{1-p'_{j}}(x)dx\bigg)^{\frac{1}{p_{i}s_{i,j}}}.$$
Therefore,
\begin{align*}
\prod_{i=1}^{m}\|\nu_{\vec{\omega}}^{1-m}\chi_{Q}\|_{L^{p_{i}}(\omega_{i})}&\leq \prod_{j=1}^{m}\bigg(\int_{Q}\omega_{j}^{1-p'_{j}}(x)dx\bigg)^{\frac{(m-1)p}{p'_{j}(mp-1)}}\lesssim |Q|^{\frac{mp(m-1)}{mp-1}}\nu_{\vec{\omega}}(Q)^{\frac{m-1}{1-mp}}.
\end{align*}
Combining the result in Lemma \ref{weights-lem2}, we have $\nu_{\vec{\omega}}^{\alpha}\chi_{Q}\in \bigcap_{i=1}^{m}L^{p_{i}}(\omega_{i})$ when $\frac{1}{1-mp}\leq \alpha\leq 1-m.$

Next, we deal with the opposite case. Let
\begin{equation*}
   \omega_{1}(x):= \left\{
   \begin{array}{ll}\vspace{1ex}
|x|^{-n}, & \alpha\geq 0,\\
|x|^{-n(1-mp'_{1})(1-p_{1})+\frac{p_{1}}{mp}\epsilon}, &1-m<\alpha<0,\\
|x|^{n(p_{1}-1)-\frac{p_{1}}{mp}\epsilon},& \alpha<\frac{1}{1-mp}
   \end{array}
 \right.
\end{equation*}
with $\epsilon>0$ small enough. From the fact that $|x|^{\delta}\in A_{p}$ if and only if $-n<\delta<n(p-1)$, we have
\begin{equation*}
   \omega_{1}(x)^{1-p'_{1}}_{1}= \left\{
   \begin{array}{ll}\vspace{1ex}
|x|^{n(p'_{1}-1)}\in A_{mp'_{1}}, & \alpha\geq 0,\\
|x|^{n(mp'_{1}-1)-\frac{p'_{1}}{2mp}\epsilon}\in A_{mp'_{1}}, &1-m<\alpha<0,\\
|x|^{-n+\frac{p'_{1}}{mp}\epsilon}\in A_{mp'_{1}}, & \alpha<\frac{1}{1-mp}.
   \end{array}
 \right.
\end{equation*}
On the other hand, for $j=2,\cdots,m$, we consider
\begin{equation*}
   \omega_{j}(x):= \left\{
   \begin{array}{ll}\vspace{1ex}
1, & \alpha\geq 0,\\
|x|^{n(p_{j}-1)-\frac{p_{j}}{mp}\epsilon}, &1-m<\alpha<0 \quad \text{and} \quad \alpha<\frac{1}{1-mp}.
   \end{array}
 \right.
\end{equation*}
Then,
\begin{equation*}
   \omega_{j}(x)^{1-p'_{j}}= \left\{
   \begin{array}{ll}\vspace{1ex}
1 \in A_{mp'_{j}}, & \alpha\geq 0,\\
|x|^{-n+\frac{p'_{j}}{mp}\epsilon}\in A_{mp'_{j}}, &1-m<\alpha<0 \quad \text{and} \quad \alpha<\frac{1}{1-mp}.
   \end{array}
 \right.
\end{equation*}
By the definition of $\nu_{\vec{\omega}}=\prod_{i=1}^{m}\omega_{i}^{p/p_{i}}$, we arrive at
\begin{equation*}
  \nu_{\vec{\omega}}(x)= \left\{
   \begin{array}{ll}\vspace{1ex}
|x|^{-\frac{pn}{p_{1}}}\in A_{mp}, & \alpha\geq 0,\\
|x|^{-n+\frac{(2m-3)}{2m}\epsilon}\in A_{mp}, & 1-m<\alpha<0,\\
|x|^{n(mp-1)-\epsilon}\in A_{mp}, & \alpha<\frac{1}{1-mp},
   \end{array}
 \right.
\end{equation*}
which yields that $\vec{\omega}\in A_{\vec{P}}$ from Lemma \ref{weights-lem1}.

We now prove that for any cube $ Q=Q(o,r),$
$$\nu_{\vec{\omega}}^{\alpha}\chi_{Q}\notin L^{p_{1}}(\omega_{1}).$$
If $\alpha\geq 0$, by the direct computation, we get
$$\|\nu_{\vec{\omega}}^{\alpha}\chi_{Q}\|^{p_{1}}_{L^{p_{1}}(\omega_{1})}=\int_{Q}|x|^{-n(\alpha p+1)}dx=\infty.$$
If $1-m<\alpha<0$, we have
$$-n(1-mp'_{1})(1-p_{1})-n\alpha p_{1}<-n.$$
This gives us that
$$\|\nu_{\vec{\omega}}^{\alpha}\chi_{Q}\|^{p_{1}}_{L^{p_{1}}(\omega_{1})}=\int_{Q}|x|^{-n(1-mp'_{1})(1-p_{1})-n\alpha p_{1}+\big(\frac{(2m-3)\alpha p_{1}}{2m}-\frac{p'_{1}}{2mp}\big)\epsilon}dx=\infty,$$
when $\epsilon>0$ small enough.
If $\alpha<\frac{1}{1-mp},$ we get
$$\|\nu_{\vec{\omega}}^{\alpha}\chi_{Q}\|_{L^{p_{i}}(\omega_{i})}=\Big(\int_{Q}|x|^{-np_{i}\alpha(mp-1)+n(p_{i}-1)-\alpha p_{i}\epsilon-\frac{p_{i}\epsilon}{mp}}dx\Big)^{1/p_{i}}=\infty.$$
since $-np_{1}\alpha(mp-1)+n(p_{1}-1)<-n$.

Therefore, $\nu_{\vec{\omega}}^{\alpha}\chi_{Q}\in \bigcap_{i=1}^{m}L^{p_{i}}(\omega_{i})$ if and only if $\frac{1}{1-mp}\leq \alpha\leq 1-m.$

(2) The condition $p>\frac{1}{m-1}$ implies that $1-m<\frac{1}{1-mp}$, then, we obtain the desired result using the same arguments as above.
\end{proof}

\thanks{{\bf Acknowledgements.}}
The author would like to thank Professor Kangwei Li for the very valuable observations which have helped clarify and improve our presentation.

\end{document}